\documentclass[a4paper,11pt,fleqn]{article}
\usepackage{amsmath,amssymb,amsthm,amsfonts}  % 数学相关宏包
\usepackage{graphicx,epstopdf}                % 图片处理宏包
\usepackage{array,booktabs,tabularx}          % 表格相关宏包
\usepackage{subfigure,caption,float}          % 图片排版宏包
\usepackage{color,bm,stmaryrd}                % 颜色、加粗符号、特殊符号宏包
\usepackage{tikz}                             % 绘图宏包
\usepackage[top=1in, bottom=1in, left=1.25in, right=1.25in]{geometry}  % 页面设置
\usepackage[hidelinks]{hyperref}              % 超链接宏包
\usepackage{cleveref}                         % 智能引用宏包

\usetikzlibrary{calc,intersections,math}

\usepackage{amsmath}  % 最后加载一次amsmath以确保覆盖可能的冲突（可选，通常无需重复）
\allowdisplaybreaks[4]  % 允许公式跨页

\newtheorem{theorem}{Theorem}[section]
\newtheorem{lemma}[theorem]{Lemma}

\newtheorem{remark}[theorem]{Remark}
\newtheorem{assumption}[theorem]{Assumption}

\numberwithin{equation}{section}

\newcommand{\normmm}[1]{{\left\|\kern-0.25ex\left\|\kern-0.25ex\left\| #1
		\right\|\kern-0.25ex\right\|\kern-0.25ex\right\|}}

\begin{document}
\title{A $\mathcal{CR}$-rotated $Q_1$ nonconforming finite element method for Stokes interface problems on local anisotropic fitted mixed meshes} 
\author{
	Chenchen Geng$^{1}$ \quad
	Hua Wang$^{1}$\footnote{Correspondence author. E-mail addresses: wanghua@xtu.edu.cn (H. Wang).\\
%		202321511151@smail.xtu.edu.cn (C. Geng),\\
%		202121511212@smail.xtu.edu.cn (F. Zou),\\
%		wanghua@xtu.edu.cn (H. Wang).\\
		\textbf{Funding:} This research is supported by NSFC project 12101526 and Young Elite Scientists Sponsorship Program by CAST 2023QNRC001.}
		 \quad
	Fengren Zou$^{1}$\\
	{\small $^1$ School of Mathematics and Computational Science, Xiangtan University, Xiangtan 411105, China}\\
}
	\date{}
	\maketitle
	
\begin{abstract}
	We propose a new nonconforming finite element method for solving Stokes interface problems. The method is constructed on local anisotropic mixed meshes, which are generated by fitting the interface through simple connection of intersection points on an interface-unfitted background mesh, as introduced in \cite{Hu2021optimal}. For triangular elements, we employ the standard $\mathcal{CR}$ element; for quadrilateral elements, a new rotated $Q_1$-type element is used. We prove that this rotated $Q_1$ element remains unisolvent and stable even on degenerate quadrilateral elements. Based on these properties, we further show that the space pair of $\mathcal{CR}$-rotated $Q_1$ elements (for velocity) and piecewise $P_0$ spaces (for pressure) satisfies the inf-sup condition without requiring any stabilization terms. As established in our previous work \cite{Wang2025nonconforming}, the consistency error achieves the optimal convergence order without the need for penalty terms to control it. Finally, several numerical examples are provided to verify our theoretical results.
\end{abstract}

	
	\section{Introduction}\label{sec:introduction}
%介绍问题以及应用背景
We consider the following Stokes interface problem in a convex polygonal $\Omega$ in $\mathbb{R}^2$ (see Figure \ref{fig:mesh} (a) for an illustration)
\begin{equation}
	\begin{aligned}
			-\nabla\!\cdot\!(\mu \nabla \boldsymbol{u}-p\boldsymbol{I})& = \boldsymbol{f}\quad \mathrm{in} \; \Omega_1 \cup \Omega_2,\\
			\nabla\!\cdot\!\boldsymbol{u}&=0 \quad \mathrm{in}
			\; \Omega_1 \cup \Omega_2,\\
			[\![\boldsymbol{u}]\!] &=\boldsymbol{0}\quad 
			\mathrm{on}\; \Gamma,\\
			[\![(\mu\nabla\boldsymbol{u}-p\boldsymbol{I})\!\cdot\!\boldsymbol{n}_{\Gamma}]\!]&=\boldsymbol{0}\quad \mathrm{on}\;\Gamma,\\
			\boldsymbol{u} & =\boldsymbol{0}\quad \mathrm{on}\; \partial \Omega,
	\end{aligned}\label{eq:stokespro}
\end{equation}
where $\boldsymbol{f} \in (L^2(\Omega))^2$, $\boldsymbol{n}_{\Gamma}$ is the unit normal vector of the interface $\Gamma$ orienting from $\Omega_1$ towards $\Omega_2$, $[\![\cdot]\!]$ denotes the jump across $\Gamma$, i.e., $[\![\boldsymbol{v}]\!]=(\boldsymbol{v}_1-\boldsymbol{v}_2)|_{\Gamma}$ with $\boldsymbol{v}_i = \boldsymbol{v}|_{\Omega_i},\;i=1,2,$ and $\mu$ is a piecewise positive constant vicosity function in $\Omega$, i.e., 
\begin{equation*}
		\begin{aligned}
			\mu = \begin{cases}
				\mu_1 \quad \mathrm{in}\; \Omega_1,\\	
				\mu_2 \quad  \mathrm{in} \;\Omega_2.
			\end{cases}
		\end{aligned}
\end{equation*}

Stokes interface problems arise primarily from two-phase incompressible flows, which are prevalent in engineering and scientific computing. These problems are typically modeled using Navier-Stokes equations with discontinuous viscosity coefficients. When the viscosity of the two-phase flow is high, the Stokes interface problem—characterized by discontinuous viscous coefficients—serves as a reasonable simplification of such models.

% Unfitted Methods
Unfitted mesh methods exhibit particular effectiveness in addressing interface problems with complex geometries, owing to their flexibility in handling irregular interface structures. Significant advancements have been made in these methods, focusing on enhancing accuracy, computational efficiency, and adaptability to intricate interfaces. Notable examples include the Immersed Finite Element Method (IFEM) and the Extended Finite Element Method (XFEM), which have been extensively studied and validated in numerous works (see, e.g., \cite{Li1998, Li2003, An2014, Chou2012immersed, Lin2015, Ji2023, Wang2019a, Zhang2019strongly}). Additionally, other promising approaches such as the Generalized Finite Element Method (GFEM) have been proposed, with an overview provided in \cite{Fries2010extended}. IFEM typically modifies finite element basis functions to explicitly satisfy interface conditions, whereas XFEM introduces penalization terms into the variational formulation to weakly enforce these conditions—an approach known as interior penalty or Nitsche's methods (see \cite{Hansbo2002unfitted, Li2003, burman2015cutfem, Hansbo2014cut, Cao2022extended, Fries2010extended}). For instance, Chen et al. \cite{Chen2023an} combined XFEM with a novel mesh generation strategy, effectively merging small interface elements with neighboring elements.

% Refined Mesh Strategies
Another widely explored strategy involves refining the unfitted mesh near interfaces to construct locally fitted or anisotropic meshes. Previous studies have demonstrated significant progress using this approach (see \cite{Chen2009the, Xu2016optimal, chen2017interface, Hu2021optimal}). Chen et al. \cite{Chen2009the} generated intermediate fitted meshes by subdividing interface tetrahedra into smaller ones via the latest vertex bisection algorithm, preserving mesh quality throughout adaptive refinement. Xu et al. \cite{Xu2016optimal} proposed linear finite element schemes for diffusion and Stokes equations on interface-fitted grids satisfying the maximal angle condition. Similarly, Chen et al. \cite{chen2017interface} developed methods for semi-structured, interface-fitted mesh generation in two and three dimensions, leveraging virtual element methods to solve elliptic interface problems.

However, refined elements adjacent to interfaces often violate the minimal angle condition (shape regularity), complicating error analysis and numerical stability. Despite these challenges, this refinement approach remains prevalent due to its adaptability in handling complex interface geometries.
% Challenges for Nonconforming Elements
Most unfitted methods face significant challenges when incorporating nonconforming elements. Firstly, the consistency error cannot be adequately controlled. In Nitsche-type XFEM approaches, the weak continuity across cut edges is compromised, necessitating penalty terms to stabilize consistency errors (see \cite{Wang2019a}). For the Immersed Finite Element Method (IFEM), although weak continuity is preserved, inherent solution singularities at interfaces lead to a half-order degradation in consistency errors compared to interpolation errors (see \cite{Ji2023analysis}).

Secondly, the inf-sup condition is generally not satisfied naturally, requiring additional stabilization terms, as demonstrated in \cite{Wang2019a} and \cite{Huang2024nonconforming}. While the work in \cite{Jones2021class} proposed the $CR$-$P_0$ and rotated $Q_1$-$Q_0$ elements for solving Stokes interface problems, it did not provide a proof of the inf-sup condition.

% Our Motivation and Related Work
In prior work \cite{Wang2025nonconforming}, the authors proposed a $P_1$-nonconforming element for second-order elliptic interface problems. However, the finite element pair formed with $P_0$ elements fails the inf-sup condition due to insufficient velocity space degrees of freedom. We extend this work naturally by enriching the velocity space with edge-based degrees of freedom along the discrete interface $\Gamma_h$, while preserving the weak continuity of nonconforming elements.

The remainder of this paper is organized as follows. Section \ref{sec:preliminaries} introduces fundamental definitions and notations essential to our framework. Section \ref{sec:NFEM} then develops the nonconforming finite element method: Subsection \ref{subsec:RNE} constructs the rotated $Q_1$ element on quadrilateral meshes and establishes its stability; Subsection \ref{subsec:WF} formulates the continuous and discrete weak formulations while specifying regularity assumptions; Subsection \ref{subsec:inf-sup} provides a rigorous proof of the inf-sup condition; and Subsection \ref{subsec:prior} derives a priori error estimates for the proposed method. Section \ref{sec:example} presents numerical experiments that validate the theoretical results.

\section{Notation and preliminaries}\label{sec:preliminaries}
	For integer $r\geq 0$, define the piecewise $H^{r}$ Sobolev space
\begin{equation*}
	H^{r}(\Omega_{1}\cup\Omega_{2})=\{v\in L^2(\Omega);v|_{\Omega_{i}}\in H^{r}(\Omega_{i}),i=1,2\},
\end{equation*}
equipped with the norm and semi-norm
\begin{eqnarray*}\label{norm}
	\begin{aligned}
		\|v\|_{H^{r}(\Omega_{1}\cup\Omega_{2})}&=(\|v\|_{H^{r}(\Omega_1)}^{2}
		+\|v\|_{H^{r}(\Omega_2)}^{2})^{1/2},\\
		|v|_{H^{r}(\Omega_{1}\cup\Omega_{2})}&=(|v|_{H^{r}(\Omega_1)}^{2}
		+|v|_{H^{r}(\Omega_2)}^{2})^{1/2}.
	\end{aligned}
\end{eqnarray*}
Furthermore, let $\tilde{H}^{r}(\Omega_1\cup\Omega_2)=H^{1}_{0}(\Omega)\cap H^{r}(\Omega_{1}\cup\Omega_{2}).$
%In the case $p=2$, we use $\tilde{H}^{r}(\Omega_{1}\cup\Omega_{2})$ to represent $\tilde{W}^{r,2}(\Omega_{1}\cup\Omega_{2})$.
\begin{figure}[ht]
	\centering
	% ———— (a) 单位正方形与圆 ————
	\begin{minipage}[b]{0.31\textwidth}
		\centering
		\textbf{(a)}\\[1ex]
		\begin{tikzpicture}[scale=3.5]
			\draw[black, thick] (0,0) rectangle (1,1);
			\draw[blue, thick] (0.5,0.5) circle[radius=0.3];
			\node at (0.2,0.2) {$\Omega_{1}$};
			\node at (0.5,0.5) {$\Omega_{2}$};
			\node[blue] at (0.5,0.82) {$\Gamma$};
		\end{tikzpicture}
		\label{subfig:domain}
	\end{minipage}
	\hfill
	% ———— (b) 网格 + 圆 + 界面单元高亮 ————
	\begin{minipage}[b]{0.31\textwidth}
		\centering
		\textbf{(b)}\\[1ex]
		\begin{tikzpicture}[scale=3.5]
			\def\n{8}
			\pgfmathsetmacro{\h}{1/\n}
			% 高亮被圆穿过的三角形
			\foreach \i in {0,...,7} {
				\foreach \j in {0,...,7} {
					\pgfmathsetmacro{\x}{\i/\n}
					\pgfmathsetmacro{\y}{\j/\n}
					% 定义顶点
					\coordinate (A) at (\x,\y);
					\coordinate (B) at (\x+\h,\y);
					\coordinate (C) at (\x+\h,\y+\h);
					\coordinate (D) at (\x,\y+\h);
					% 水平集 φ = (x-0.5)^2+(y-0.5)^2-0.3^2
					\pgfmathsetmacro{\phiA}{(\x-0.5)^2+(\y-0.5)^2-0.09}
					\pgfmathsetmacro{\phiB}{(\x+\h-0.5)^2+(\y-0.5)^2-0.09}
					\pgfmathsetmacro{\phiC}{(\x+\h-0.5)^2+(\y+\h-0.5)^2-0.09}
					\pgfmathsetmacro{\phiD}{(\x-0.5)^2+(\y+\h-0.5)^2-0.09}
					% 三角形 ABC
					\pgfmathsetmacro{\minABC}{min(min(\phiA,\phiB),\phiC)}
					\pgfmathsetmacro{\maxABC}{max(max(\phiA,\phiB),\phiC)}
					\ifdim\minABC pt<0pt
					\ifdim\maxABC pt>0pt
					\fill[yellow!60] (A) -- (B) -- (C) -- cycle;
					\fi
					\fi
					% 三角形 ACD
					\pgfmathsetmacro{\minACD}{min(min(\phiA,\phiC),\phiD)}
					\pgfmathsetmacro{\maxACD}{max(max(\phiA,\phiC),\phiD)}
					\ifdim\minACD pt<0pt
					\ifdim\maxACD pt>0pt
					\fill[yellow!60] (A) -- (C) -- (D) -- cycle;
					\fi
					\fi
				}
			}
			% 黑色网格与对角线
			\foreach \i in {0,...,7} {
				\foreach \j in {0,...,7} {
					\pgfmathsetmacro{\x}{\i/\n}
					\pgfmathsetmacro{\y}{\j/\n}
					\draw[black, thin] (\x,\y) rectangle ++(\h,\h);
					\draw[black, thin] (\x,\y) -- ++(\h,\h);
				}
			}
			% 最上层画圆
			\draw[blue, thick] (0.5,0.5) circle[radius=0.3];
			% 区域标记
			\node[inner sep=0.8pt] at (0.2,0.85) {$\mathcal{T}_h^{N}$};
			\node[inner sep=0.8pt] at (0.75,0.75) {$\mathcal{T}_h^{\Gamma}$};
		\end{tikzpicture}
		\label{subfig:unfitted}
	\end{minipage}
	\hfill
	% ———— (c) 多边形逼近并标记 Γ_h ————
	\begin{minipage}[b]{0.31\textwidth}
		\centering
		\textbf{(c)}\\[1ex]
		\begin{tikzpicture}[scale=3.5]
			\def\n{8}
			\pgfmathsetmacro{\h}{1/\n}
			% 黑色网格
			\foreach \i in {0,...,7} {
				\foreach \j in {0,...,7} {
					\pgfmathsetmacro{\x}{\i/\n}
					\pgfmathsetmacro{\y}{\j/\n}
					\draw[black, thin] (\x,\y) rectangle ++(\h,\h);
					\draw[black, thin] (\x,\y) -- ++(\h,\h);
				}
			}
			% 圆轮廓
			%\draw[blue, thick] (0.5,0.5) circle[radius=0.3];
			% 凸多边形逼近
			\draw[red, thick]
			(0.8,0.5)  -- (0.75,0.65) -- (0.65,0.75) -- (0.5,0.8)  --
			(0.35,0.75) -- (0.25,0.65) -- (0.2,0.5)  -- (0.25,0.35) --
			(0.35,0.25) -- (0.5,0.2)   -- (0.65,0.25) -- (0.75,0.35) -- cycle;
			% 标记多边形为 Γ_h
			\node[red] at (0.5,0.85) {$\Gamma_h$};
		\end{tikzpicture}
		\label{subfig:fitted}
	\end{minipage}
	
	\caption{Geometric interface and mesh interaction:
		(a) the computational domain for the interface problem;
		(b) unfitted mesh~$\mathcal{T}_h$;
		(c) local anisotropic hybrid mesh~$\tilde{\mathcal{T}}_h$.}
	\label{fig:mesh}
\end{figure}
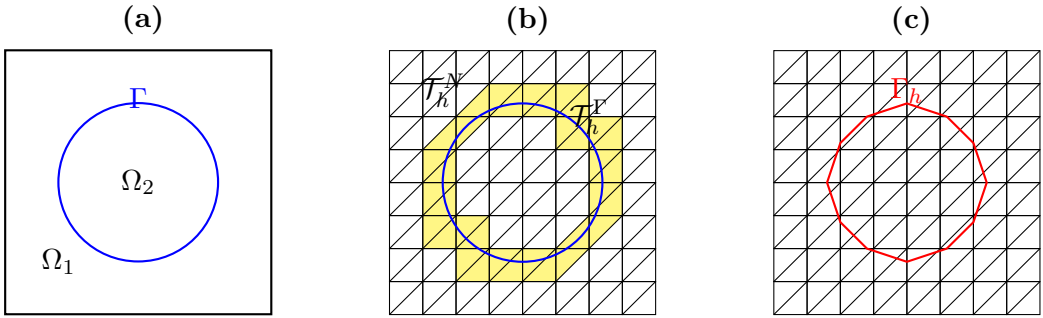	

We initiate the process by generating an interface-unfitted mesh $\mathcal{T}_h$, which serves as the background mesh (see Figure \ref{fig:mesh}(b)). By sequentially connecting the intersection points of the interface $\Gamma$ (blue line) and the mesh edges, a polygonal approximation $\Gamma_h$ (red line) of the interface $\Gamma$ is constructed. The resulting mesh, denoted by $\widetilde{\mathcal{T}}_h$ (see Figure \ref{fig:mesh}(c)), is an interface-fitted mesh that contains anisotropic triangles and quadrilaterals in the vicinity of the interface. The domain $\Omega$ is thereby partitioned into two polygonal subdomains $\Omega_{1,h}$ and $\Omega_{2,h}$ by $\Gamma_h$, which serve as approximations to $\Omega_1$ and $\Omega_2$, respectively.

Define the following mesh subsets:
\begin{align}
	&\mathcal{T}_h^{\Gamma} := \{K \in \mathcal{T}_h \,;\, K \cap \Gamma \neq \emptyset\},\\
	&\mathcal{T}_h^N := \mathcal{T}_h \setminus \mathcal{T}_h^{\Gamma}. 
\end{align}
Elements in $\mathcal{T}_h^{\Gamma}$ are referred to as interface elements. The mesh $\widetilde{\mathcal{T}}_h$ can be regarded as a refinement of $\mathcal{T}_h$. Let $\widetilde{\mathcal{E}}_h$ denote the set of all edges in $\widetilde{\mathcal{T}}_h$. Define $\widetilde{\mathcal{T}}_{h,i}$ as the subset of elements in $\widetilde{\mathcal{T}}_h$ that lie within $\Omega_{i,h}$. Let $\widetilde{\mathcal{E}}_h^{\Gamma}$ denote the collection of edges that coincide with $\Gamma_h$, and $\widetilde{\mathcal{E}}_h^N := \widetilde{\mathcal{E}}_h \setminus \widetilde{\mathcal{E}}_{h}^{\Gamma}$. Additionally, we denote the set of boundary edges by $\widetilde{\mathcal{E}}_h^0$.

\section{The nonconforming method}\label{sec:NFEM}
\subsection{The rotated nonconforming element space}\label{subsec:RNE}
%define space
%unisolvence
%interpolation error
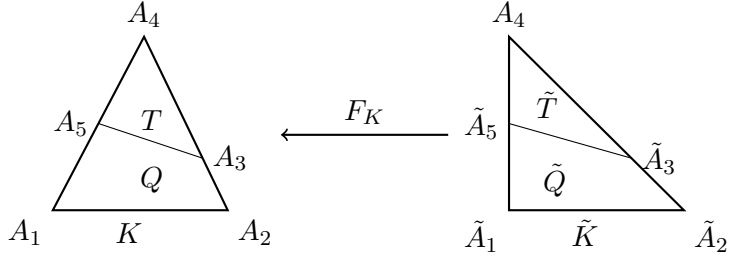
\begin{figure}[ht]
	\centering
	\begin{tikzpicture}[scale=1]
		% —— 左侧：锐角三角形 K ——
		\coordinate (A1) at (0,0);
		\coordinate (A2) at (2.3,0);
		\coordinate (A4) at (1.2,2.3);
		\draw[thick] (A1) -- (A2) -- (A4) -- cycle;
		\node at (1,-0.3) {$K$};
		
		% 分割点：A3 在 A2A4 边上靠近 A2 的三等分点，A5 在 A1A4 边上靠近 A4 的四等分点
		\coordinate (A3) at ($(A2)!0.3!(A4)$);
		\coordinate (A5) at ($(A1)!0.5!(A4)$);
		\draw[thin] (A3) -- (A5);
		
		% 标注顶点
		\node[below left]  at (A1) {$A_1$};
		\node[below right] at (A2) {$A_2$};
		\node[above]       at (A4) {$A_4$};
		\node[right]       at (A3) {$A_3$};
		\node[left]        at (A5) {$A_5$};
		
		% （可选）标注子区域 T 和 Q
		\node at (1.3,1.2) {$T$};
		\node at (1.3,0.4) {$Q$};
		
		% —— 箭头：仿射变换 F_K ——
		\draw[->, thick](5.2,1)  -- node[above] {$F_K$} (3,1);
		
		% —— 右侧：直角三角形 \hat K ——
		\begin{scope}[xshift=6cm]
			\coordinate (Ah1) at (0,0);
			\coordinate (Ah2) at (2.3,0);
			\coordinate (Ah4) at (0,2.3);
			\draw[thick] (Ah1) -- (Ah2) -- (Ah4) -- cycle;
			\node at (1,-0.3) {$\tilde K$};
			
			% 分割点
			\coordinate (Ah3) at ($(Ah2)!0.3!(Ah4)$);
			\coordinate (Ah5) at ($(Ah1)!0.5!(Ah4)$);
			\draw[thin] (Ah3) -- (Ah5);
			
			% 标注顶点
			\node[below left]  at (Ah1) {$\tilde A_1$};
			\node[below right] at (Ah2) {$\tilde A_2$};
			\node[above]       at (Ah4) {$\tilde A_4$};
			\node[right]       at (Ah3) {$\tilde A_3$};
			\node[left]        at (Ah5) {$\tilde A_5$};
			
			% （可选）标注子区域 \hat T 和 \hat Q
			\node at (0.5,1.4) {$\tilde T$};
			\node at (0.6,0.4) {$\tilde Q$};
		\end{scope}
	\end{tikzpicture}
	\caption{The interface macro element}
	\label{fig:macro-element}
\end{figure}
The following discussion concerns the construction and properties of basis functions on interface elements, along with the associated interpolation error estimates. Consider a general interface element $K \in \mathcal{T}_h^{\Gamma}$ as illustrated in Figure~\ref{fig:macro-element}. Define the cut ratio parameters by
\begin{equation*}
	t = \frac{|A_1A_5|}{|A_1A_4|}, \quad s =  \frac{|A_2A_3|}{|A_2A_4|}.
\end{equation*}
Clearly, $0 \leq s, t < 1$. Without loss of generality, we assume $s \leq t$; otherwise, we apply a reflection transformation to satisfy this condition. Note that when $t = 1$, $\tilde{A}_5$ coincides with $\tilde{A}_4$, which we exclude to avoid degeneracy. An affine mapping $\bm{F}_K$ maps the physical interface element $K$ to a reference element $\tilde{K}$:
\begin{equation}
	\tilde{x} = \bm{F}_K(\tilde{\bm{x}}) = \bm{B} \tilde{\bm{x}} + \bm{b}.
\end{equation}
The reference coordinates of the vertices $\tilde{A}_1, \dots, \tilde{A}_5$ are
\begin{equation*}
	\tilde{A}_1 = (0, 0),\quad \tilde{A}_2 = (1, 0),\quad \tilde{A}_3 = (1 - s, s),\quad \tilde{A}_4 = (0, 1),\quad \tilde{A}_5 = (0, t).
\end{equation*}
Under reasonable assumptions (see Assumption 3.1 in \cite{Ji2023}), we consider the following configurations: 

\textbf{Case I.} The interface passes through a vertex of $K$ ($s = 0$). In this case, $K$ is divided into two triangles, both satisfying the maximum angle condition (see \cite{Hu2021optimal}). This case is straightforward to handle, as standard CR elements are used for both sub-triangles.

\textbf{Case II.} The interface intersects the interior of two edges of $K$ ($0 < s \leq t < 1$). In this case, $K$ is divided into a triangle and a quadrilateral. For this case, we utilize standard CR elements for the triangular sub-element and a rotated $Q_1$-type element for the quadrilateral sub-element. Since the quadrilateral sub-element may be anisotropic and potentially degenerate, the construction of a stable rotated $Q_1$ element constitutes a primary focus of this work.
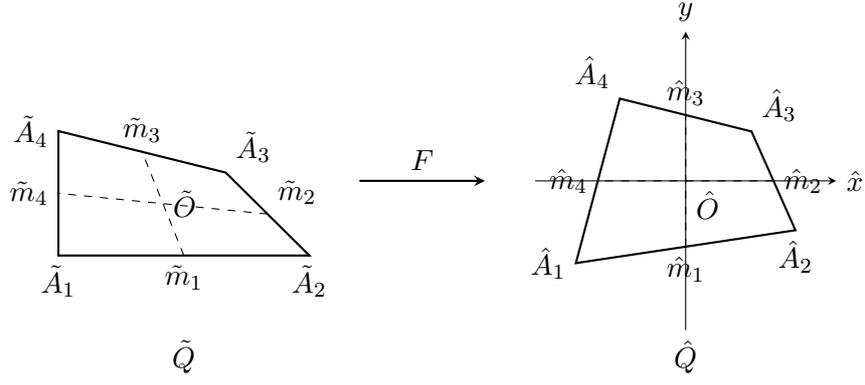
\begin{figure}[ht] 
	\centering
	\begin{tikzpicture}[scale=3.3]
		\def\s{1/3}  % s 值
		\def\t{1/2}  % t 值
		\def\k{1.4}
		\begin{scope}[yshift=-0.3cm]
			% 设置参数
			% ====== 左侧图：原始四边形 ======
			% 顶点坐标
			\coordinate (A1) at (0,0);
			\coordinate (A2) at (1,0);
			\coordinate (A3) at (1-\s,\s);
			\coordinate (A4) at (0,\t);
			
			% 计算中点
			\coordinate (m1) at ($(A1)!0.5!(A2)$);
			\coordinate (m2) at ($(A2)!0.5!(A3)$);
			\coordinate (m3) at ($(A3)!0.5!(A4)$);
			\coordinate (m4) at ($(A1)!0.5!(A4)$);
			\coordinate (O) at ($(m1)!0.5!(m3)$);
			
			% 绘制四边形
			\draw[thick] (A1) -- (A2) -- (A3) -- (A4) -- cycle;
			
			% 绘制虚线连接
			\draw[dashed] (m1) -- (m3);
			\draw[dashed] (m2) -- (m4);
			
			% 标记顶点
			\node[below] at (A1) {$\tilde{A}_1$};
			\node[below] at (A2) {$\tilde{A}_2$};
			\node[above right] at (A3) {$\tilde{A}_3$};
			\node[left] at (A4) {$\tilde{A}_4$};
			
			% 标记中点
			\node[below] at (m1) {$\tilde{m}_1$};
			\node[above right] at (m2) {$\tilde{m}_2$};
			\node[above] at (m3) {$\tilde{m}_3$};
			\node[left] at (m4) {$\tilde{m}_4$};
			\node[right] at (O) {$\tilde{O}$};
			\node[below] at (1/2,-0.3) {$\tilde{Q}$};
		\end{scope}
		
		\draw[->, thick, >=stealth] (1.2,0) -- node[above] { $F$} (1.7,0);
		% ====== 右侧图：仿射变换后的四边形 ======
		\begin{scope}[xshift=2.5cm]
			% 预计算分母
			\pgfmathsetmacro{\denomOne}{sqrt(\t/2 - \s/2 + (\s/2 - 1)^2)}
			\pgfmathsetmacro{\denomTwo}{sqrt(\s^2/4 + \s/2 + \t/2)}
			
			% 计算变换后的中点坐标
			\pgfmathsetmacro{\hatmOneX}{0}
			\pgfmathsetmacro{\hatmOneY}{-(\t/4 - \s/4 + (\t*(\s/2 - 1))/2)/\denomOne}
			
			\pgfmathsetmacro{\hatmTwoX}{-(\s/4 + \t/4 - (\s*\t)/4)/\denomTwo}
			\pgfmathsetmacro{\hatmTwoY}{0}
			
			\pgfmathsetmacro{\hatmThreeX}{0}
			\pgfmathsetmacro{\hatmThreeY}{((\s*(\s/2 - 1))/2 - (\s/2 - \t/2)*(\s/2 - 1/2))/\denomOne}
			
			\pgfmathsetmacro{\hatmFourX}{-((\s/2 + \t/2)*(\s/2 - 1/2) - \s^2/4)/\denomTwo}
			\pgfmathsetmacro{\hatmFourY}{0}
			
			% 计算变换后的顶点坐标
			\pgfmathsetmacro{\hatAOneX}{-((\s/2 + \t/2)*(\s/2 - 1/2) - (\s*(\s/2 + \t/2))/2)/\denomTwo}
			\pgfmathsetmacro{\hatAOneY}{-(\t*(\s/2 - 1))/(2*\denomOne)}
			
			\pgfmathsetmacro{\hatATwoX}{-((\s/2 + \t/2)*(\s/2 + 1/2) - (\s*(\s/2 + \t/2))/2)/\denomTwo}
			\pgfmathsetmacro{\hatATwoY}{-(\t/2 - \s/2 + (\t*(\s/2 - 1))/2)/\denomOne}
			
			\pgfmathsetmacro{\hatAThreeX}{((\s/2 + \t/2)*(\s/2 - 1/2) - (\s*(\s/2 - \t/2))/2)/\denomTwo}
			\pgfmathsetmacro{\hatAThreeY}{-((\s/2 - \t/2)*(\s - 1) - (\s/2 - 1)*(\s - \t/2))/\denomOne}
			
			\pgfmathsetmacro{\hatAFourX}{-((\s/2 + \t/2)*(\s/2 - 1/2) - (\s*(\s/2 - \t/2))/2)/\denomTwo}
			\pgfmathsetmacro{\hatAFourY}{(\t*(\s/2 - 1))/(2*\denomOne)}
			
			% 定义变换后的点
			\coordinate (hatA1) at (-\k*\hatAOneX, -\k*\hatAOneY);
			\coordinate (hatA2) at (-\k*\hatATwoX, -\k*\hatATwoY);
			\coordinate (hatA3) at (-\k*\hatAThreeX, -\k*\hatAThreeY);
			\coordinate (hatA4) at (-\k*\hatAFourX, -\k*\hatAFourY);
			
			\coordinate (hatm1) at (-\k*\hatmOneX, -\k*\hatmOneY);
			\coordinate (hatm2) at (-\k*\hatmTwoX, -\k*\hatmTwoY);
			\coordinate (hatm3) at (-\k*\hatmThreeX, -\k*\hatmThreeY);
			\coordinate (hatm4) at (-\k*\hatmFourX, -\k*\hatmFourY);
			
			% 绘制坐标轴
			\draw[->, >=stealth] (-0.6,0) -- (0.6,0) node[right] {$\hat{x}$};
			\draw[->, >=stealth] (0,-0.6) -- (0,0.6) node[above] {$\hat{y}$};
			
			% 绘制变换后的四边形
			\draw[thick] (hatA1) -- (hatA2) -- (hatA3) -- (hatA4) -- cycle;
			
			% 绘制虚线连接
			\draw[dashed] (hatm1) -- (hatm3);
			\draw[dashed] (hatm2) -- (hatm4);
			
			% 标记变换后的点
			\node[left] at (hatA1) {$\hat{A}_1$};
			\node[below] at (hatA2) {$\hat{A}_2$};
			\node[above right] at (hatA3) {$\hat{A}_3$};
			\node[above left] at (hatA4) {$\hat{A}_4$};
			
			% 标记变换后的中点
			\node[below] at (hatm1) {$\hat{m}_1$};
			\node[right] at (hatm2) {$\hat{m}_2$};
			\node[above] at (hatm3) {$\hat{m}_3$};
			\node[left] at (hatm4) {$\hat{m}_4$};
			
			% 添加标签
			\node[below right] at (0,0) {$\hat{O}$};
			\node[below] at (0,-0.6) {$\hat{Q}$};
		\end{scope}
	\end{tikzpicture}\label{fig:transformation}
	\caption{An affine map from a reference quadrilateral $\hat{Q}$ to a quadrilateral $\tilde{Q}$}
\end{figure}

%------------------------------
By direct calculation, the equations for lines $\tilde{m}_1\tilde{m}_3$ and $\tilde{m}_2\tilde{m}_4$ are:
\begin{align}
	l_{13}(\tilde{x},\tilde{y}) &= \frac{s + t}{|\tilde{Q}|}(\tilde{x} - \frac{1}{2}) + \frac{s}{|\tilde{Q}|}\tilde{y}, \label{eq:l13} \\
	l_{24}(\tilde{x},\tilde{y}) &= \frac{t - s}{|\tilde{Q}|}\tilde{x} + \frac{2 - s}{|\tilde{Q}|}\tilde{y} - \frac{t(2 - s)}{2|\tilde{Q}|}, \label{eq:l24}
\end{align}
where $|\tilde{Q}| = \frac{1}{2}(s + t - st)$ denotes the area of $\tilde{Q}$. An affine transformation $\bm{F}$ maps points $\tilde{m}_i$ to reference points $\hat{m}_i$ at $(0,-1)$, $(1,0)$, $(0,1)$, and $(-1,0)$ (see Figure \ref{fig:transformation}):
\begin{equation*}
	\begin{pmatrix} \hat{x} \\ \hat{y} \end{pmatrix} = 
	\begin{pmatrix} l_{13}(\tilde{x},\tilde{y}) \\ l_{24}(\tilde{x},\tilde{y}) \end{pmatrix}.
\end{equation*}
This transformation yields the vertex coordinates:
\begin{equation*}\label{eq:transformed-vertices}
	\hat{A}_1(-c_1, -c_2), \quad \hat{A}_2(c_1, -2+c_2), \quad \hat{A}_3(2 - c_1, 2-c_2), \quad \hat{A}_4(-2+c_1, c_2),
\end{equation*}
with $c_1 = \dfrac{s + t}{2|\tilde{Q}|}$ and $c_2 = \dfrac{(2 - s)t}{2|\tilde{Q}|}$. It is easy to verify that these coefficients satisfy the bounds:
\begin{equation*}
	1 \leq c_1 \leq 2, \quad \frac{1}{2} \leq c_2 \leq 2.	
\end{equation*}
The finite element triple $(\hat{Q}, \mathcal{P}_{\hat{Q}}, \mathcal{N})$ is defined as:
\begin{equation}\label{def:FE}
	\mathcal{P}_{\hat{Q}} = P_1 \oplus \{\hat{x}^2\}, \quad
	\mathcal{N} = \{\mathcal{N}_1, \dots, \mathcal{N}_4\}, \quad \mathcal{N}_i(\hat{v}) = \frac{1}{|\hat{e}_i|}\int_{\hat{e}_i}\hat{v}  ds.
\end{equation}
The corresponding Vandermonde matrix is:
\begin{equation}
	\bm{M} = \begin{pmatrix}
		1 &  0 & -1 &  c_1^2/3 \\
		1 &  1 &  0 &  (c_1 - 2)^2/6 + c_1^2/6 + 2/3 \\
		1 &  0 &  1 &  (c_1 - 2)^2/3 \\
		1 & -1 &  0 &  (c_1 - 2)^2/6 + c_1^2/6 + 2/3 
	\end{pmatrix}.
\end{equation}

\begin{lemma}
	The degrees of freedom $\{\mathcal{N}_i\}_{i=1}^4$ are unisolvent for $\mathcal{P}_{\hat{Q}}$.
\end{lemma}
\begin{proof}
	It suffices to show that the homogeneous system has only the trivial solution. Assume $p(\hat{x},\hat{y}) = \alpha_0 + \alpha_1 \hat{x} + \alpha_2 \hat{y} + \alpha_3 \hat{x}^2$ 
	satisfies $\mathcal{N}_i(p) = 0$ for $1 \leq i \leq 4$.	since the determinant of the Vandermonde matrix $\bm{M}$ is $8/3 \neq 0$, thus $\alpha_0 = \alpha_1 = \alpha_2 = \alpha_3 = 0$.
\end{proof}

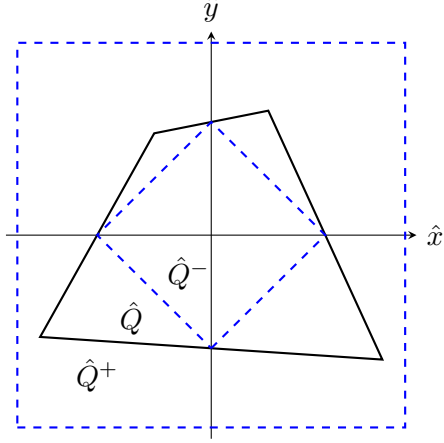
\begin{figure}[ht] 
	\centering
	\begin{tikzpicture}[scale=1.5]
		% 预计算参数
		\def\a{1.5}  % s 值
		\def\b{0.9}  % t 值	
		
		% 定义变换后的点
		\coordinate (hatA1) at (-\a,-\b);
		\coordinate (hatA2) at (\a,\b-2);
		\coordinate (hatA3) at (2-\a,2-\b);
		\coordinate (hatA4) at (\a-2, \b);
		
		\coordinate (hatAA1) at (-1.7, -1.7);
		\coordinate (hatAA2) at (1.7,-1.7);
		\coordinate (hatAA3) at (1.7,1.7);
		\coordinate (hatAA4) at (-1.7,1.7);
		
		\coordinate (hatm1) at (0,-1);
		\coordinate (hatm2) at (1,0);
		\coordinate (hatm3) at (0,1);
		\coordinate (hatm4) at (-1,0);
		
		% 绘制坐标轴
		\draw[->, >=stealth] (-1.8,0) -- (1.8,0) node[right] {$\hat{x}$};
		\draw[->, >=stealth] (0,-1.8) -- (0,1.8) node[above] {$\hat{y}$};
		
		% 四边形为实线
		\draw[thick, solid] (hatA1) -- (hatA2) -- (hatA3) -- (hatA4) -- cycle;
		\draw[blue,thick, dashed] (hatAA1) -- (hatAA2) -- (hatAA3) -- (hatAA4) -- cycle; 
		
		% 绘制内部正方形（中点连接）
		\draw[blue, thick,dashed] (hatm1) -- (hatm2) -- (hatm3) -- (hatm4) -- cycle;

		% 添加标签
		\node[below] at (-0.2,-0.1) {$\hat{Q}^-$};
		\node[below] at (-0.7,-0.5) {$\hat{Q}$};
		\node[below] at (-1,-1) {$\hat{Q}^+$};
	\end{tikzpicture}\label{fig:hatQ}
	\caption{The reference quadrilateral $\hat{Q}$}
\end{figure}
Therefore, we can define the nonconforming element space on the locally anisotropic hybrid mesh $\widetilde{\mathcal{T}}_h$ as
\begin{equation}\label{def:velocity-space}
	U_h = \left\{ v \in L^2(\Omega) \;\middle|\; v|_K \in \mathcal{P}(K)\ \forall K \in \widetilde{\mathcal{T}}_h,\ \int_e [v]\,\mathrm{d}s = 0\ \forall e \in \widetilde{\mathcal{E}}_h \right\}.
\end{equation}
If $K$ is a triangle element $\mathcal{P}(K)=\mathcal{P}_1(K)$. If $K$ is a quadrilateral element, $\mathcal{P}(K)$ is defined as the pullback of $\mathcal{P}_{\hat{Q}}$ as in \eqref{def:FE}. Define the piecewise $H^1$ semi-norm as
\begin{equation}\label{eq:energy-norm}
	\| v_h \|_{U_h} = (\sum_{K \in \mathcal{T}_h} \int_K |\nabla v_h|^2  dx)^{1/2}.
\end{equation}
This defines a norm on $U_h$ due to the Dirichlet boundary conditions.

Note that the reference quadrilateral $\hat{Q}$ is inscribed in a larger square $\hat{Q}^+$ with side length $4$, while its interior contains a smaller square $\hat{Q}^-$ with side length $\sqrt{2}$ (see Figure \ref{fig:hatQ}). Therefore, For any $\hat{v} \in \mathcal{P}_k(\hat{Q})$, we have
\begin{equation}\label{bound_hatQ}
\begin{aligned}	
	|\hat{v}|_{H^r(\hat{Q}^-)} \lesssim |\hat{v}|_{H^r(\hat{Q})} \lesssim |\hat{v}|_{H^r(\hat{Q}^+)}.
\end{aligned}
\end{equation}
By \eqref{bound_hatQ}, we derive the following lemma which means the basis functions $\hat{\phi}_i$ defined on the nonstandard reference element $\hat{Q}$ (see Definition \ref{def:FE}) satisfy the same order bounds as those on the reference square element:
\begin{lemma}\label{lem:basis_bound}
	For basis functions $\hat{\phi}_i$ defined by \eqref{def:FE}, we have
	\begin{align}
		|\hat{\phi}_i|_{H^r(\hat{Q})} \lesssim 1\quad r =0,1, 
	\end{align} 
	where the hidden constant is independent of the element geometry.
\end{lemma}
\begin{proof}
	By definition, $\hat{\phi}_i = \bm{M}^{-1} \bm{\delta}_i^{\top} (1, \hat{x}, \hat{y}, \hat{x}^2)$, where $\bm{M}$ is the Vandermonde matrix for the basis. 
	For any $\hat{p}\in \{1, \hat{x}, \hat{y}, \hat{x}^2\}$, since
	\begin{equation*}
		\|\hat{p}\|_{L^2(\hat{Q})}\lesssim \|\hat{p}\|_{L^2(\hat{Q}^+)}\lesssim 1,
	\end{equation*}
	it suffices to prove $\|\bm{M}^{-1} \bm{\delta}_i\|_{\infty} \lesssim 1$. The adjugate matrix $\bm{M}^*$ satisfies $\|\bm{M}^*\|_{\infty} = \mathcal{O}(1)$ as its entries are cofactors of $\bm{M}$. Thus,
	\[
	\|\bm{M}^{-1}\|_{\infty} = \left\| \frac{1}{\det(\bm{M})}\bm{M}^* \right\|_{\infty} \lesssim 1.
	\]
	Consequently, $\|\bm{M}^{-1} \bm{\delta}_i\|_{\infty} \lesssim \|\bm{M}^{-1}\|_{\infty} \|\bm{\delta}_i\|_{\infty} \lesssim 1$.
\end{proof}

The following lemma establishes the scaling relations for the $H^1$-seminorm between $\tilde{Q}$ and $\hat{Q}$, which provides upper bounds for the $H^1$-seminorm of basis functions on $\tilde{Q}$. These bounds will be essential for the inf-sup analysis later.

\begin{lemma}\label{lem:affine-est}
	For $\hat{v} \in H^1(\hat{Q})$, the partial derivatives satisfy anisotropy-adaptive estimates:
	\begin{align}
		\|\partial_{\tilde{x}} \tilde{v}\|_{L^2(\tilde{Q})} &\lesssim t^{1/2} |\hat{v}|_{H^1(\hat{Q})}, \\
		\|\partial_{\tilde{y}} \tilde{v}\|_{L^2(\tilde{Q})} &\lesssim t^{-1/2} |\hat{v}|_{H^1(\hat{Q})},
	\end{align}
	where the hidden constants are independent of $t$ and $s$.
\end{lemma}
\begin{proof}
	Applying the chain rule, we have
	\begin{equation}
		\tilde{\nabla} \tilde{v} = \bm{D\!F}^\top \hat{\nabla} \hat{v}.
	\end{equation}
	where 
	\begin{align}
		\bm{D\!F} &= \frac{1}{|\tilde{Q}|} \begin{pmatrix}
				t + s & s \\
				t - s & 2 - s
			\end{pmatrix}, \label{eq:jacobian} 
	\end{align}
	are the Jacobian matrix of the affine mapping $F$. For the $x$-derivative:
	\begin{align*}
		\|\partial_{\tilde{x}} \tilde{v}\|_{L^2(\tilde{Q})} 
		&= \left\| \frac{(t + s)\partial_{\hat{x}} \hat{v} + (t - s)\partial_{\hat{y}} \hat{v}}{|\tilde{Q}|} \right\|_{L^2(\tilde{Q})} \\
		&\lesssim |\det(\bm{D\!F})|^{-1/2} \left( |t + s| \|\partial_{\hat{x}} \hat{v}\|_{L^2(\hat{Q})} + |t - s| \|\partial_{\hat{y}} \hat{v}\|_{L^2(\hat{Q})} \right) \\
		&\lesssim t^{1/2} |\hat{v}|_{H^1(\hat{Q})}.
	\end{align*}
	Similarly for the $y$-derivative:
	\begin{align*}
		\|\partial_{\tilde{y}} \tilde{v}\|_{L^2(\tilde{Q})} 
		&= \left\| \frac{s\partial_{\hat{x}} \hat{v} + (2 - s)\partial_{\hat{y}} \hat{v}}{|\tilde{Q}|} \right\|_{L^2(\tilde{Q})} \\
		&\lesssim |\det(\bm{D\!F})|^{-1/2} \left( |s| \|\partial_{\hat{x}} \hat{v}\|_{L^2(\hat{Q})} + |2 - s| \|\partial_{\hat{y}} \hat{v}\|_{L^2(\hat{Q})} \right) \\
		&\lesssim t^{-1/2} |\hat{v}|_{H^1(\hat{Q})}.
	\end{align*}
\end{proof}

\begin{remark}\label{rmk:bubble-choice}
	From an approximation perspective, the bubble function for the nonconforming element on $\hat{Q}$ can be any quadratic term except $\hat{x}\hat{y}$. However, our analysis of the inf-sup condition requires interpolation stability within the function space. As established in Lemma \ref{lem:affine-est}, interpolation stability fails in the $y$-direction due to element anisotropy.	To address this issue, we select $\hat{x}^2$ as the quadratic term. This choice ensures that the shape function space is linear in $\hat{y}$, allowing us to leverage the special properties of linear functions to prove interpolation stability. 
	
	On the other hand, according to Apel's work \cite{Apel2001cr}, it is necessary to select the square of the variable corresponding to the axis of the relatively longer edge as the quadratic term. In the quadrilateral $\tilde{Q}$, it is readily observed that $|\tilde{m}_1\tilde{m}_3| \lesssim |\tilde{m}_2\tilde{m}_4|$. Therefore, the shape function space on $\tilde{Q}$ is expressed as $P_1 \oplus \{l_{13}^2\}$, which is the pullback of $\mathcal{P}_{\hat{Q}}$ to $\tilde{Q}$.
\end{remark}

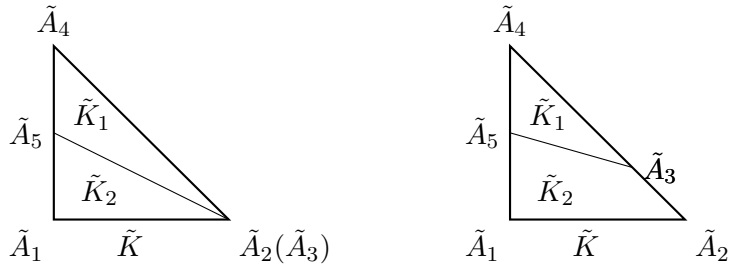
\begin{figure}[ht]
	\centering
	\begin{tikzpicture}[scale=1]
		% —— 左侧：锐角三角形 K ——
			\coordinate (Ah1) at (0,0);
            \coordinate (Ah2) at (2.3,0);
            \coordinate (Ah4) at (0,2.3);
            \draw[thick] (Ah1) -- (Ah2) -- (Ah4) -- cycle;
            \node at (1,-0.3) {$\tilde K$};

           % 分割点
           %\coordinate (Ah3) at ($(Ah2)!0.3!(Ah4)$);
           \coordinate (Ah5) at ($(Ah1)!0.5!(Ah4)$);
           \draw[thin] (Ah2) -- (Ah5);

          % 标注顶点
          \node[below left]  at (Ah1) {$\tilde A_1$};
          \node[below right] at (Ah2) {$\tilde A_2 (\tilde{A}_3)$};
          \node[above]       at (Ah4) {$\tilde A_4$};
          \node[right]       at (Ah3) {$\tilde A_3$};
          \node[left]        at (Ah5) {$\tilde A_5$};

          % （可选）标注子区域 \hat T 和 \hat Q
          \node at (0.5,1.4) {$\tilde K_1$};
          \node at (0.6,0.4) {$\tilde K_2$};
		
		% —— 右侧：直角三角形 \hat K ——
		\begin{scope}[xshift=6cm]			
		    \coordinate (Ah1) at (0,0);
			\coordinate (Ah2) at (2.3,0);
			\coordinate (Ah4) at (0,2.3);
			\draw[thick] (Ah1) -- (Ah2) -- (Ah4) -- cycle;
			\node at (1,-0.3) {$\tilde K$};
			
			% 分割点
			\coordinate (Ah3) at ($(Ah2)!0.3!(Ah4)$);
			\coordinate (Ah5) at ($(Ah1)!0.5!(Ah4)$);
			\draw[thin] (Ah3) -- (Ah5);
			
			% 标注顶点
			\node[below left]  at (Ah1) {$\tilde A_1$};
			\node[below right] at (Ah2) {$\tilde A_2$};
			\node[above]       at (Ah4) {$\tilde A_4$};
			\node[right]       at (Ah3) {$\tilde A_3$};
			\node[left]        at (Ah5) {$\tilde A_5$};
			
			% （可选）标注子区域 \hat T 和 \hat Q
			\node at (0.5,1.4) {$\tilde K_1$};
			\node at (0.6,0.4) {$\tilde K_2$};
		\end{scope}
	\end{tikzpicture}
	\caption{Two type interface macro element}
	\label{fig:macro-element}
\end{figure}
Let $\phi_{\Gamma,e}$ be the basis functions defined which degree of freedom is defined on the discrete interface edge $e\in \mathcal{E}_h^{\Gamma}$, see Figure \ref{fig:macro-element}. The following lemma  gives an upper bound for $\tilde{\phi}_{\Gamma,\tilde{e}}:= \phi_{\Gamma,e}\circ F_{K}$.
\begin{lemma}\label{lem:basis-est}
	Let $M$ be one of the subelement of the interface macro element, it holds that
	\begin{align}
		\|\partial_{\tilde{x}} \tilde{\phi}_{\Gamma,e}\|_{L^2(\tilde{K})} \lesssim |K|^{1/2},\\
		\|\partial_{\tilde{y}} \tilde{\phi}_{\Gamma,e}\|_{L^2(\tilde{K})} \lesssim |K|^{-1/2}.
	\end{align} 
\end{lemma}
\begin{proof}
For case I, the results follows by a direct calculation.
For case II, let $K= $
\end{proof}

\subsection{weak formulations}\label{subsec:WF}
%contiunous and discrete weak formulation
%regularity assumption(proof?)
Let $\boldsymbol{U} = (H^1_0(\Omega))^2$ and $X = L^2_0(\Omega)$. 
The continuous weak formulation of the Stokes interface problem is: 
find $(\boldsymbol{u},p) \in \boldsymbol{U} \times X$ satisfying
\begin{equation}\label{eq:stweak}
	\begin{cases}
		a(\boldsymbol{u}, \boldsymbol{v}) - b(\boldsymbol{v}, p) = (\bm{f},\boldsymbol{v}) & \forall \boldsymbol{v} \in \boldsymbol{U}, \\
		b(\boldsymbol{u}, q) = 0 & \forall q \in X,
	\end{cases}
\end{equation}
where
\begin{align*}
	a(\boldsymbol{u}, \boldsymbol{v}) = \int_{\Omega} \mu \nabla \boldsymbol{u} \cdot \nabla \boldsymbol{v}  dx, \quad
	b(\boldsymbol{v}, p) = \int_{\Omega} (\nabla \cdot \boldsymbol{v}) p  dx.
\end{align*}

The pressure finite element space uses piecewise constants:
\begin{equation}\label{eq:pressure-space}
	X_h = \{ q \in L^2_0(\Omega) \mid q|_K \in \mathcal{P}_0(K) \ \forall K \in \mathcal{T}_h \}.
\end{equation}

Using the velocity space $\boldsymbol{U}_h$ defined in \eqref{def:velocity-space} and pressure space \eqref{eq:pressure-space}, 
the discrete variational formulation is: find $(\boldsymbol{u}_h, p_h) \in \boldsymbol{U}_h \times X_h$ satisfying
\begin{equation}\label{eq:stweak-discrete}
	\begin{cases}
		a_h(\boldsymbol{u}_h, \boldsymbol{v}_h) - b_h(\boldsymbol{v}_h, p_h) = (\bm{f},\boldsymbol{v}_h) & \forall \boldsymbol{v}_h \in \boldsymbol{U}_h, \\
		b_h(\boldsymbol{u}_h, q_h) = 0 & \forall q_h \in X_h,
	\end{cases}
\end{equation}
with discrete forms
\begin{align*}
	a_h(\boldsymbol{u}_h, \boldsymbol{v}_h) = \sum_{K \in \widetilde{\mathcal{T}}_h} \int_K \mu_h \nabla \boldsymbol{u}_h \cdot \nabla \boldsymbol{v}_h  dx,\quad
	b(\boldsymbol{v}_h, p_h) = \sum_{K \in \widetilde{\mathcal{T}}_h}\int_{K} (\nabla \cdot \boldsymbol{v}) p_h  dx.
\end{align*}
where $\mu_h$ is the discrete approximation of the piecewise constant coefficient $\mu$ defined as
\[
\mu_h|_{K} =
\begin{cases}
	\mu_1 & K \in \mathcal{T}_{h,1}, \\
	\mu_2 & K \in \mathcal{T}_{h,2},
\end{cases}
\]

For the error analysis, we make the following regularity assumption:
\begin{assumption}\label{ass:regularity}
	Assume the interface $\Gamma$ is $C^2$-smooth, $\boldsymbol{f} \in L^2(\Omega)^2$, 
	and the solution $(\boldsymbol{u}, p)$ of the Stokes interface problem satisfies:
	\begin{align*}
		\boldsymbol{u} \in \boldsymbol{\tilde{H}}^2(\Omega_1 \cup \Omega_2), \quad	p \in H^1(\Omega_1 \cup \Omega_2),
	\end{align*}
	with the regularity estimate:
	\begin{equation}
		\|\boldsymbol{u}\|_{H^2(\Omega_1 \cup \Omega_2)} + \|p\|_{H^1(\Omega_1 \cup \Omega_2)} \lesssim \|\boldsymbol{f}\|_{L^2(\Omega)}.
	\end{equation}
\end{assumption}
 
\subsection{inf-sup condition}\label{subsec:inf-sup}
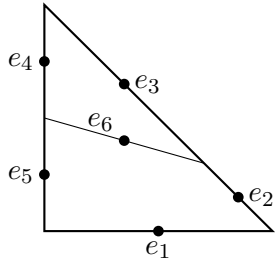
\begin{figure}[ht]
	\centering
	\begin{tikzpicture}[scale=1]
		% —— 直角三角形 \hat{K} ——
		\coordinate (Ah1) at (0,0);
		\coordinate (Ah2) at (3,0);
		\coordinate (Ah4) at (0,3);
		\draw[thick] (Ah1) -- (Ah2) -- (Ah4) -- cycle;
		
		% —— 分割线 ——
		\coordinate (Ah3) at ($(Ah2)!0.3!(Ah4)$);
		\coordinate (Ah5) at ($(Ah1)!0.5!(Ah4)$);
		\draw[thin] (Ah3) -- (Ah5);
		
		% —— 中点 ——
		\coordinate (Mh1) at ($(Ah1)!0.5!(Ah2)$);
		\coordinate (Mh2) at ($(Ah2)!0.5!(Ah3)$);
		\coordinate (Mh3) at ($(Ah3)!0.5!(Ah4)$);
		\coordinate (Mh4) at ($(Ah4)!0.5!(Ah5)$);
		\coordinate (Mh5) at ($(Ah5)!0.5!(Ah1)$);
		\coordinate (Mh6) at ($(Ah3)!0.5!(Ah5)$);
		
		% —— 绘制中点并凸显 ——
		% 实心小点
		\foreach \M in {Mh1,Mh2,Mh3,Mh4,Mh5,Mh6} {\fill (\M) circle (2pt);}
		% 空心小点
		%	\draw (Mh6) circle (2pt);
		
		% —— 标注中点 ——
		\node[below]      at (Mh1) {$ e_1$};
		\node[right]      at (Mh2) {$ e_2$};
		\node[right]      at (Mh3) {$ e_3$};
		\node[left]       at (Mh4) {$ e_4$};
		\node[left]       at (Mh5) {$ e_5$};
		\node[above left] at (Mh6) {$ e_6$};
		
		% （可选）标注子区域 \hat T 和 \hat Q
		%		\node at (0.5,2.0) {$\hat T$};
		%		\node at (0.6,0.7) {$\hat Q$};
	\end{tikzpicture}
	\caption{Degrees of freedom}
	\label{fig:degree}
\end{figure}
In this chapter, we discuss the stability of $\boldsymbol{U}_h\times X_h,$ and we use the space decomposition technique to show that $\boldsymbol{U}_h\times X_h$ satisfies the discrete inf-sup condition. 
%因为标准的CR-P0元是满足inf-sup条件的，所以我们这里只需要考虑界面单元。假设M是一个界面宏单元，我们做如下空间分解。
Let $N_I$ be the number of interface elements in $\mathcal{T}_h^{\Gamma}.$ For $K\in \mathcal{T}_h^{\Gamma}$, the interface element is seprated into a triangle $T$ and a quadrilateral $Q$, we define a function $q_{h,i}\in X_h$ by
\begin{flalign*}
	&&
	\begin{aligned}
		q_{h,i}
		= \begin{cases}
			|T|^{-1} \quad ~~~~ &\mathrm{in}\, T,\\
			-|Q|^{-1} \quad  &\mathrm{in} \,Q,\\
			0 \quad &\mathrm{in} \,\Omega\!\setminus\! M.
		\end{cases}
	\end{aligned}
	&&
\end{flalign*}
For the case that the interface element is divided into two triangles, the definition of $q_{h,i}$ is similar. We decompose the $X_h$ space as
\begin{equation}
X_h=X_{0,h}\oplus X^\bot _{0,h},
\end{equation}
where $X_{0,h}= \mathrm{span}\{q_{h,i}\}^{N_I}_{i=1}$. Through a simple calculation, we have
\begin{equation}
	X^{\perp}_{0,h}=\{q_h\in X|\,q_h|_K\in P_0,\,\forall K\in \widetilde{\mathcal{T}}_h\}
\end{equation}
  Let $\{\phi_i\}^6_{i=1}$ be a set of basis functions on each edge $\{e_i\}^6_{i=1}$ of local finite element space $U_h(M)$ on interface element $M$ in 
\Cref{fig:degree}, satisfying
$$\frac{1}{|e_i|}\int_{e_i}\phi_j\,ds =\delta_{i,j}.$$
Let $U_{0,h}=\mathrm{span}\{\phi_{\Gamma,e}\}^{N_I}_{j=1},$ $\boldsymbol{U}_{0,h}=U_{0,h}\times U_{0,h}.$ Through simple calculations, we know that $\boldsymbol{U}_{0,h}\subset \boldsymbol{U}_h.$ First prove that $\boldsymbol{U}_{0,h}\times X_{0,h}$ satisfies the inf-sup condition.

Before proceeding, we introduce a lemma. While the standard result holds globally on $\Omega$ with a domain-dependent constant, the following lemma establishes a local version on a single element where the constant $C$ is independent of the element size.
\begin{lemma}\label{le:div}
	For interface elements $M \in \mathcal{T}_h^{\Gamma}$, there exists a positive constant $C$ independent of $h$, such that for all $p \in L^2_0(M)$, there exists $\bm{v} \in (H^1_0(M))^2$ satisfying:
	\begin{align}
		\nabla \cdot \bm{v} &= p \quad \mathrm{in}\, M, \\
		|\bm{v}|_{H^1(M)} &\leq C \|p\|_{L^2(M)},
	\end{align}
	where $C$ depends only on the maximum angle condition of the triangulation.
\end{lemma}

\begin{proof}
	Let $\hat{M}$ be the reference triangle element and $F_M: \hat{M} \to M$ the affine mapping:
	\begin{equation*}
		\bm{x} = \bm{B}\hat{\bm{x}} + \bm{b}. 
	\end{equation*}
	For $p \in L^2_0(M)$, define the scaled function on the reference element $\hat{p} = p\circ F_M$.	Note that $\hat{p} \in L^2_0(\hat{M})$ since $\int_{\hat{M}} \hat{p}  d\hat{\bm{x}} = |\det \bm{B}|\int_M p  d\bm{x} = 0$. It is well known that there exists $\hat{\bm{v}} \in (H^1_0(\hat{M}))^2$ such that:
	\begin{align}
		\hat{\nabla} \cdot \hat{\bm{v}} &= \hat{p} \quad \text{in } \hat{M}, \label{eq:ref-div} \\
		\|\hat{\bm{v}}\|_{H^1(\hat{M})} &\leq C_{\hat{M}} \|\hat{p}\|_{L^2(\hat{M})}, \label{eq:ref-stab}
	\end{align}
	where $C_{\hat{M}} = O(1)$ since $\hat{M}$ is a reference triangle. Using the Piola transformation, we define
	\begin{equation}
		\bm{v}(\bm{x}) = \bm{B} (\hat{\bm{v}}\circ F_M^{-1}).
	\end{equation}
Therefore, we derive
	\begin{align*}
		\nabla \cdot \bm{v}
		&= (\hat{\nabla} \cdot \hat{\bm{v}})\circ F_M^{-1} \\
		&= \hat{p}\circ F_M^{-1} \\
		&= p.
	\end{align*}
Since $\hat{\bm{v}} = \bm{0}$ on $\partial\hat{M}$, we have $\bm{v} = \bm{0}$ on $\partial M$. The gradient transforms as:
	\begin{equation*}
		\nabla \bm{v} =  \bm{B} \bm{B}^{-\top}(\hat{\nabla} \hat{\bm{v}}) .
	\end{equation*}
	Thus,
	\begin{align*}
		|\bm{v}|_{H^1(M)}^2 
		&= \int_M |\nabla \bm{v}|^2 d\bm{x} \\
		&\lesssim \|\bm{B}\|^2 \|\bm{B}^{-1}\|^2 |\det\bm{B}|\int_{\hat{M}} |\hat{\nabla} \hat{\bm{v}}|^2 d\hat{\bm{x}} \\
		&\lesssim h_M^2\|\hat{\bm{v}}\|_{H^1(\hat{M})}^2\\
		&\lesssim h_M^2\|\hat{p}\|_{L^2(\hat{M})}^2\\
		&\lesssim \|p\|_{L^2(M)}^2.
	\end{align*}
The proof completes.
\end{proof}

%----------------
The following lemma shows that the finite element pair $\bm{U}_{0,h}\times X_{0,h}$ satisfies the inf-sup condition.
\begin{lemma}\label{le:sup1}
	There exists a constant $k_1 > 0$, independent of $h$, such that for any $q_{0,h} \in X_{0,h}$, it holds
	\begin{equation}
		k_1\| q_{0,h} \|_{L^2(\Omega)} \leq \sup_{\substack{\boldsymbol{v}_{1,h} \in \boldsymbol{U}_{0,h} \\ \boldsymbol{v}_{1,h} \neq \bm{0}}} \frac{b_h(\boldsymbol{v}_{1,h}, q_{0,h})}{\|\boldsymbol{v}_{1,h}\|_{\boldsymbol{U}_{h}}}.
	\end{equation}
\end{lemma}

\begin{proof}
	For any $q_{0,h} \in X_{0,h}$ and any macro-element $M \in \mathcal{T}_h^{\Gamma}$, since $\int_M q_{0,h}  dx = 0$, Lemma \ref{le:div} guarantees the existence of $\boldsymbol{v}_1 \in \bm{H}^1_0(M)$ satisfying:
	\begin{align*}
		\nabla \cdot \boldsymbol{v}_1 &= q_{0,h}\quad \text{in} ~M, \\
		|\boldsymbol{v}_1|_{H^1(M)} &\lesssim \|q_{0,h}\|_{L^2(M)}.
	\end{align*}
	
	Define the interpolation operator $\Pi^{(1)}_M: \bm{H}^1(M) \rightarrow \bm{U}_{h}(M)$ by
	\begin{equation}\label{eq:pi_m^1}
		\int_e \Pi^{(1)}_M \boldsymbol{v}  ds = \int_e \boldsymbol{v}  ds, \quad \forall e \in \mathcal{E}(M).
	\end{equation}
	This implies the representation:
	\[
	\Pi^{(1)}_M \boldsymbol{v} = \sum_{i=1}^{6} \left( \frac{1}{|e_i|} \int_{e_i} \boldsymbol{v}  ds \right) \boldsymbol{\phi}_i.
	\]
	For $\boldsymbol{v}_1 \in \bm{H}^1_0(M)$, we have specifically:
	\[
	\Pi^{(1)}_M \boldsymbol{v}_1 = \left( \frac{1}{|e_6|} \int_{e_6} \boldsymbol{v}_1  ds \right) \boldsymbol{\phi}_6.
	\]
	
	We first establish the stability on the triangular sub-element $T$:
	Since $\Pi^{(1)}_M \boldsymbol{v}|_T \in \bm{P}_1(T)$, its gradient is constant. By Green's formula:
	\[
	\int_T \partial_x (\Pi^{(1)}_M \boldsymbol{v})  dx = \int_T \partial_x \boldsymbol{v}  dx.
	\]
	Thus,
	\[
	\partial_x (\Pi^{(1)}_M \boldsymbol{v})|_T = \frac{1}{|T|} \int_T \partial_x \boldsymbol{v}  dx.
	\]
	The $L^2$-norm satisfies:
	\begin{align*}
		\|\partial_x (\Pi^{(1)}_M \boldsymbol{v})\|_{L^2(T)} 
		&= |T|^{1/2} \left| \frac{1}{|T|} \int_T \partial_x \boldsymbol{v}  dx \right| \\
		&\leq |T|^{-1/2} \left| \int_T \partial_x \boldsymbol{v}  dx \right| \\
		&\lesssim \|\partial_x \boldsymbol{v}\|_{L^2(T)}.
	\end{align*}
	Similarly, $\|\partial_y (\Pi^{(1)}_M \boldsymbol{v})\|_{L^2(T)} \lesssim \|\partial_y \boldsymbol{v}\|_{L^2(T)}$. Therefore,
	\begin{equation}\label{eq:pi_m^1err}
		|\Pi^{(1)}_M \boldsymbol{v}|_{H^1(T)} \lesssim |\boldsymbol{v}|_{H^1(T)}.
	\end{equation}
	
	For the quadrilateral sub-element $Q$, using scaling argument, we have
	\begin{equation*}
		|\Pi^{(1)}_M \boldsymbol{v}|_{H^1(Q)} \lesssim |\tilde{\Pi}^{(1)}_{\tilde{M}} \boldsymbol{\tilde{v}}|_{H^1(\tilde{Q})}.
	\end{equation*}
    and 
    \[
    \tilde{\Pi}^{(1)}_{\tilde{M}} \boldsymbol{\tilde{v}}_1 = \left( \frac{1}{|\tilde{e}_6|} \int_{\tilde{e}_6} \boldsymbol{\tilde{v}}_1  ds \right) \boldsymbol{\tilde{\phi}}_6.
    \]
    since $\partial_{\tilde{y}} (\tilde{\Pi}^{(1)}_{\tilde{M}} \boldsymbol{\tilde{v}}_1)$ is constant in $\tilde{Q}$, using the same argument as for the triangle subelement, we have 
    \[
    \|\partial_{\tilde{y}} (\tilde{\Pi}^{(1)}_{\tilde{M}} \boldsymbol{\tilde{v}}_1)\|_{L^2(Q)} \lesssim \|\partial_{\tilde{y}} \boldsymbol{\tilde{v}}_1\|_{L^2(\tilde{Q})}.
    \]
    
    Using lemma \ref{lem:basis-est}, we have
    \begin{align*}
    	|\partial_{\tilde{x}}\tilde{\phi}_6|_{H^1(\tilde{Q})}\lesssim |\tilde{Q}|^{1/2}.
    \end{align*}
    consequently,  let $\delta \boldsymbol{\tilde{v}}_1 = \boldsymbol{\tilde{v}}_1 - \frac{1}{|\tilde{Q}|}\int_{\tilde{Q}}\boldsymbol{\tilde{v}}_1 d\tilde{x}$, we derive
	\begin{align*}
		\|\partial_{\tilde{x}} (\Pi^{(1)}_{\tilde{M}} \boldsymbol{\tilde{v}}_1)\|_{L^2(\tilde{Q})} 
		&=\|\partial_{\tilde{x}} (\Pi^{(1)}_{\tilde{M}} \delta\boldsymbol{\tilde{v}}_1)\|_{L^2(\tilde{Q})} \\
		&= \left| \frac{1}{|\tilde{e}_6|} \int_{\tilde{e}_6} \delta\boldsymbol{\tilde{v}}_1  ds \right| \|\partial_{\tilde{x}} \boldsymbol{\tilde{\phi}}_6\|_{L^2(\tilde{Q})} \\
		&\lesssim \frac{|\tilde{Q}|^{1/2}}{|\tilde{e}_6|^{1/2}}\|\delta\boldsymbol{\tilde{v}}_1\|_{L^2(\tilde{e}_6)}\\
		&\lesssim \frac{|\tilde{Q}|^{1/2}}{|\tilde{e}_6|^{1/2}} \frac{|\tilde{e}_6|^{1/2}}{|\tilde{Q}|^{1/2}}(\|\delta\boldsymbol{\tilde{v}}_1\|_{L^2(\tilde{Q})} + h_{\tilde{Q}}|\delta\boldsymbol{\tilde{v}}_1|_{H^1(\tilde{Q})} )\\
		&\lesssim |\boldsymbol{\tilde{v}}_1|_{H^1(\tilde{Q})}.
	\end{align*}
	Therefore,
	\begin{equation}\label{eq:pi_m^1Q}
		|\tilde{\Pi}^{(1)}_{\tilde{M}} \boldsymbol{\tilde{v}}_1|_{H^1(\tilde{Q})} \lesssim |\boldsymbol{\tilde{v}}_1|_{H^1(\tilde{Q})}.
	\end{equation}
	and furthermore
	\begin{equation*}
		|\Pi^{(1)}_M \boldsymbol{v}|_{H^1(Q)} \lesssim |\tilde{\Pi}^{(1)}_{\tilde{M}} \boldsymbol{\tilde{v}}|_{H^1(\tilde{Q})}
		\lesssim |\boldsymbol{\tilde{v}}_1|_{H^1(\tilde{Q})}\lesssim |\bm{v}_1|_{H^1(Q)}.
	\end{equation*}
	Combining \eqref{eq:pi_m^1err} and \eqref{eq:pi_m^1Q}:
	\begin{equation}\label{eq:pi_m^1QT}
		|\Pi^{(1)}_M \boldsymbol{v}_1|_{H^1(T \cup Q)} \lesssim |\boldsymbol{v}_1|_{H^1(M)}.
	\end{equation}
	The proof completes by Fortin trick.
\end{proof}

Next, we establish the inf-sup condition for $\boldsymbol{U}_h \times X^{\perp}_{0,h}$.
\begin{lemma}\label{le:sup2}
	There exists a constant $k_2 > 0$, independent of $h$, such that for any $q^{\bot}_{0,h} \in X^{\bot}_{0,h}$, it holds
	\begin{equation}
		k_2\| q_{0,h}^{\perp} \|_{L^2(\Omega)} \leq \sup_{\substack{\boldsymbol{v}_{2,h} \in \boldsymbol{U}_{h} \\ \boldsymbol{v}_{2,h} \neq \bm{0}}} \frac{b_h(\boldsymbol{v}_{2,h}, q_{0,h}^{\perp})}{\|\boldsymbol{v}_{2,h}\|_{\boldsymbol{U}_{h}}}.
	\end{equation}
\end{lemma}

\begin{proof}
	By the inf-sup condition for the $\boldsymbol{P}_2-P_0$ element \cite[Proposition 8.4.3]{brezzi2012mixed}, for any $q^{\perp}_{0,h} \in X^{\perp}_{0,h}$, there exists $\boldsymbol{v}_2 $ belongs to the $\boldsymbol{P}_2$ Lagrange velocity space satisfying:
	\begin{equation}\label{eq:v2err}
		\nabla \cdot \boldsymbol{v}_2 = q^{\perp}_{0,h}, \quad
		|\boldsymbol{v}_2|_{H^1(\Omega)} \lesssim \|q^{\perp}_{0,h}\|_{L^2(\Omega)}.
	\end{equation}
	
	Define the local interpolation operator $\Pi^{(2)}_M: H^1(M) \rightarrow \boldsymbol{U}_h(M)$ analogously to $\Pi^{(1)}_M$ in Lemma \ref{le:sup1}. The stability estimates on the triangular sub-element $T$ follow similarly:
	\begin{equation}\label{eq:pi_m^2err}
		|\Pi^{(2)}_M \boldsymbol{v}_2|_{H^1(T)} \lesssim |\boldsymbol{v}_2|_{H^1(T)}.
	\end{equation}
	
	For the quadrilateral sub-element $Q$, 	using the scaling techqueic, we have
	\begin{align*}
		|\Pi^{(2)}\boldsymbol{v}_2|_{H^1(Q)}\lesssim |\tilde{\Pi}^{(2)}\boldsymbol{\tilde{v}}_2|_{H^1(\tilde{Q})}
	\end{align*}
	Similarly to Lemma \ref{le:sup1}, it holds
	\begin{equation*}
		\|\partial_{\tilde{y}} \tilde{\Pi}^{(2)}_{\tilde{M}} \boldsymbol{\tilde{v}}_2\|_{L^2(\tilde{Q})} \lesssim \|\partial_{\tilde{y}} \boldsymbol{\tilde{v}}_2\|_{L^2(\tilde{Q})}.
	\end{equation*}
	For the $\tilde{x}$-derivative, we derive
	\begin{align*}
		\|\partial_{\tilde{x}} \Pi^{(2)} \widetilde{\boldsymbol{v}}_2\|_{L^2(\widetilde{Q})}
		&= \inf_{\boldsymbol{c} \in \boldsymbol{P}_0} \|\partial_{\tilde{x}} \Pi^{(2)} (\widetilde{\boldsymbol{v}}_2 + \boldsymbol{c})\|_{L^2(\widetilde{Q})} \\
		&\lesssim |\widetilde{Q}|^{1/2} \inf_{\boldsymbol{c} \in \boldsymbol{P}_0} | \widehat{\Pi}^{(2)} (\widehat{\boldsymbol{v}}_2 + \boldsymbol{c})|_{H^1(\widehat{Q})} \quad \text{(by Lemma \ref{lem:affine-est})} \\
		&\lesssim |\widetilde{Q}|^{1/2} \inf_{\boldsymbol{c} \in \boldsymbol{P}_0} \sum_{i=1}^{6} \left| \frac{1}{|\hat{e}_i|} \int_{\hat{e}_i} (\widehat{\boldsymbol{v}}_2 + \boldsymbol{c}) ds \right| \cdot | \widehat{\boldsymbol{\phi}}_i|_{H^1(\widehat{Q})} \\
		&\lesssim |\widetilde{Q}|^{1/2} \inf_{\boldsymbol{c} \in \boldsymbol{P}_0} \| \widehat{\boldsymbol{v}}_2 + \boldsymbol{c} \|_{L^{\infty}(\widehat{Q})} \quad \text{(by Lemma \ref{lem:basis_bound})} \\
		&= |\widetilde{Q}|^{1/2} \inf_{\boldsymbol{c} \in \boldsymbol{P}_0} \| \widetilde{\boldsymbol{v}}_2 + \boldsymbol{c} \|_{L^{\infty}(\widetilde{Q})} \\
		&\lesssim \inf_{\boldsymbol{c} \in \boldsymbol{P}_0} \| \widetilde{\boldsymbol{v}}_2 + \boldsymbol{c} \|_{L^{\infty}(\widetilde{M})} \\
		&\lesssim \inf_{\boldsymbol{c} \in \boldsymbol{P}_0} \| \widetilde{\boldsymbol{v}}_2 + \boldsymbol{c} \|_{H^1(\widetilde{M})} 
		\quad \text{(by norm equivalence in $\bm{P}_2$)} \\
		&\lesssim | \widetilde{\boldsymbol{v}}_2 |_{H^1(\widetilde{M})} \quad \text{(by norm equivalence in quotient space)} \\
		&\lesssim | \boldsymbol{v}_2 |_{H^1(M)}.
	\end{align*}
	
	Therefore, we derive 
	\begin{equation}\label{eq:pi_M^2TQerr}
		|\Pi^{(2)}_M \boldsymbol{v}_2|_{H^1(Q)} \lesssim |\boldsymbol{v}_2|_{H^1(M)}.
	\end{equation}
	
	Jointly with \eqref{eq:pi_m^2err}, we obtain the macro-element stability:
	\begin{equation}\label{eq:pi_M^2TQ1err}
		|\Pi^{(2)}_M \boldsymbol{v}_2|_{H^1(T \cup Q)} \lesssim |\boldsymbol{v}_2|_{H^1(M)}.
	\end{equation}
	The proof completes by Fortin trick.
\end{proof}

Combining these two inf-sup lemmas for the subspaces, we now establish the inf-sup condition for the full space $\bm{U}_h\times X_h$.
\begin{theorem}\label{th:qhsup}
	There exists a constant $k>0,$ independent of $h,$ such that for any $q_{h}\in X_{h},$ it holds
	\begin{equation}\label{eq:infsup}
		k\| q_h \|_{L^2(\Omega)} \leq \sup_{\substack{\boldsymbol{v}_h \in \boldsymbol{U}_h \\ \boldsymbol{v}_h \neq \boldsymbol{0}}} \frac{b_h(\boldsymbol{v}_h,q_h)}{\|\boldsymbol{v}_h\|_{\boldsymbol{U}_h}}.
	\end{equation}
\end{theorem}

\begin{proof}
	We shall prove an equivalent condition of \eqref{eq:infsup} (see \cite{Olshanskii2006}): for any $q_h \in X_h,$ there exists $\boldsymbol{v}_h \in \boldsymbol{U}_h$ such that
	\begin{align}\label{eq:vherr}
		\|q_h\|^2_{L^2(\Omega)} &\lesssim b_h(\boldsymbol{v}_h,q_h), \\
		\|\boldsymbol{v}_h\|_{\boldsymbol{U}_h} &\lesssim \|q_h\|_{L^2(\Omega)}.
	\end{align}
	
	For $q_h \in X_h,$ decompose it as
	\begin{equation*}
		q_h = q_{0,h} + q^{\perp}_{0,h},
	\end{equation*} 
	where $q_{0,h} \in X_{0,h},$ $q^{\perp}_{0,h} \in X^{\perp}_{0,h}.$ By orthogonality,
	\begin{equation*}
		\|q_h\|^2_{L^2(\Omega)} = \|q_{0,h}\|^2_{L^2(\Omega)} + \|q^{\perp}_{0,h}\|^2_{L^2(\Omega)}.
	\end{equation*} 
	
	Define $\boldsymbol{v}_h = \boldsymbol{v}_{1,h} + \gamma\boldsymbol{v}_{2,h},$ where $\boldsymbol{v}_{1,h}$ and $\boldsymbol{v}_{2,h}$ are chosen as in Lemmas \ref{le:sup1} and \ref{le:sup2} respectively. Then
	\begin{equation*}
		b_h(\boldsymbol{v}_h,q_h) = \ b_h(\boldsymbol{v}_{1,h},q_{0,h}) + b_h(\boldsymbol{v}_{1,h},q^{\perp}_{0,h}) + \gamma b_h(\boldsymbol{v}_{2,h},q_{0,h}) + \gamma b_h(\boldsymbol{v}_{2,h},q^{\perp}_{0,h}).
	\end{equation*}	
	Since $\boldsymbol{v}_{1,h} \in \boldsymbol{U}_{0,h},$ we have
	\begin{align*}
		b_h(\boldsymbol{v}_{1,h},q^{\perp}_{0,h}) 
		= \sum_{K\in \mathcal{T}_h^{\Gamma}} \int_K \mathrm{div}\boldsymbol{\phi}^{K}_6 \cdot q^{\perp}_{0,h} \, dx 
		= \sum_{K\in \mathcal{T}_h^{\Gamma}} (q^{\perp}_{0,h}|_{K}) \int_{\partial K} \boldsymbol{\phi}^K_6 \cdot \boldsymbol{n} \, ds = 0.
	\end{align*}
	Combining Lemmas \ref{le:sup1} and \ref{le:sup2}, and using Cauchy-Schwarz inequality, we obtain:
	\begin{align*}
		b_h(\boldsymbol{v}_h, q_h) 
		&= b_h(\boldsymbol{v}_{1,h}, q_{0,h}) + \gamma b_h(\boldsymbol{v}_{2,h}, q_{0,h}) + \gamma b_h(\boldsymbol{v}_{2,h}, q_{0,h}^\perp) \\
		&= \| q_{0,h} \|_{L^2(\Omega)}^2 + \gamma b_h(\boldsymbol{v}_{2,h}, q_{0,h}) + \gamma \| q_{0,h}^\perp \|_{L^2(\Omega)}^2 \\
		&\geq \| q_{0,h} \|_{L^2(\Omega)}^2 - \gamma \| \boldsymbol{v}_{2,h} \|_{U_h} \| q_{0,h} \|_{L^2(\Omega)} + \gamma \| q_{0,h}^\perp \|_{L^2(\Omega)}^2 \\
		&\geq \| q_{0,h} \|_{L^2(\Omega)}^2 - C \gamma \| q_{0,h}^\perp \|_{L^2(\Omega)} \| q_{0,h} \|_{L^2(\Omega)} + \gamma \| q_{0,h}^\perp \|_{L^2(\Omega)}^2 \\
		&\geq \frac{\gamma}{2} \left( \| q_{0,h}^\perp \|_{L^2(\Omega)} - C \| q_{0,h} \|_{L^2(\Omega)} \right)^2 + \frac{\gamma}{2} \| q_{0,h}^\perp \|_{L^2(\Omega)}^2 \\
		&\quad + \left(1 - \frac{C^2 \gamma}{2}\right) \| q_{0,h} \|_{L^2(\Omega)}^2 \\
		&\geq \frac{1}{C^2 + 1} \| q_h \|_{L^2(\Omega)}^2,
	\end{align*}
	where $\gamma = 2/(C^2+1).$ 
	
	Furthermore, for the norm bound of $\boldsymbol{v}_h$, we derive
	\begin{align*}
		\|\boldsymbol{v}_h\|_{\boldsymbol{U}_h}
		&\lesssim \|\boldsymbol{v}_{1,h}\|_{\boldsymbol{U}_h} + \gamma \|\boldsymbol{v}_{2,h}\|_{\boldsymbol{U}_h} \\
		&\lesssim \|q_{0,h}\|_{L^2(\Omega)} + \|q^{\perp}_{0,h}\|_{L^2(\Omega)} \\
		&\lesssim \|q_h\|_{L^2(\Omega)}.
	\end{align*}
	Thus the proof completes.
\end{proof}

\subsection{An a prior error estimate}\label{subsec:prior}
The following consistency error estimate is established in \cite{Wang2025nonconforming}:
\begin{lemma}\label{le:consistency}
	Assume the solution $(\bm{u}, p)$ of the Stokes interface problem \eqref{eq:stokespro} satisfies Assumption \ref{ass:regularity}. Then the consistency error satisfies
	\begin{equation}
		E_h(\bm{u}, p,\bm{v}_h) \lesssim h (\|\bm{u}\|_{H^2(\Omega_1 \cup \Omega_2)} + \|p\|_{H^1(\Omega_1 \cup \Omega_2)}) \|\bm{v}_h\|_{\bm{U}_h}.
	\end{equation}
\end{lemma}

Leveraging the inf-sup condition established in Theorem \ref{th:qhsup}, we obtain the following error estimate:
\begin{theorem}\label{th:err}
	Let $(\boldsymbol{u},p)$ and $(\boldsymbol{u}_h,p_h)$ be solutions to problems \eqref{eq:stweak} and \eqref{eq:stweak-discrete}, respectively. Then
	\begin{equation}\label{ferr}
		\|\boldsymbol{u}-\boldsymbol{u}_h\|_{\boldsymbol{U}_h} + \|p-p_h\|_{L^2(\Omega)} \lesssim h (\|\boldsymbol{u}\|_{H^2(\Omega_1\cup\Omega_2)} + \|p\|_{H^1(\Omega_1\cup\Omega_2)}).
	\end{equation}	
\end{theorem}
\begin{proof}
	Since nonconforming elements are employed, Brezzi's theorem cannot be applied directly. Applying Green's formula yields
	\begin{equation}\label{eq:err}
		\begin{cases}
			a_h(\boldsymbol{u}-\boldsymbol{u}_h, \boldsymbol{v}_h) - b_h(\boldsymbol{v}_h, p-p_h) = E(\bm{u},p,\boldsymbol{v}_h) & \forall \boldsymbol{v}_h \in \boldsymbol{U}_h, \\
			b_h(\boldsymbol{u}-\boldsymbol{u}_h, q_h) = 0 & \forall q_h \in X_h,
		\end{cases}
	\end{equation}
	where 
	\begin{align*}
		E(\bm{u},p,\boldsymbol{v}_h) = a_h(\boldsymbol{u}, \boldsymbol{v}_h) - b_h(\boldsymbol{v}_h, p) - (\boldsymbol{f},\boldsymbol{v}_h).
	\end{align*}
	
	For any $(\bm{w}_h,z_h)\in \bm{U}_h\times X_h$, it holds that
	\begin{align*}
		a_h(\boldsymbol{u}_h-\bm{w}_h, \boldsymbol{v}_h) - b_h(\boldsymbol{v}_h, p_h-z_h) &= a_h(\boldsymbol{u}-\bm{w}_h, \boldsymbol{v}_h) - b_h(\boldsymbol{v}_h, p-z_h) - E(\bm{u},p,\boldsymbol{v}_h) \\
		b_h(\boldsymbol{u}_h-\bm{w}_h, q_h) &= b_h(\boldsymbol{u}-\bm{w}_h, q_h)
	\end{align*}
	for all $\boldsymbol{v}_h \in \boldsymbol{U}_h$ and $q_h \in X_h$. By Brezzi's theorem, we have
	\begin{align*}
		\|\boldsymbol{u}_h-\bm{w}_h\|_{\bm{U}_h} + \|p_h-z_h\|_{L^2(\Omega)} \lesssim &\|\boldsymbol{u}-\bm{w}_h\|_{\bm{U}_h} + \|p-z_h\|_{L^2(\Omega)} \\
		                                                               &+ h (\|\bm{u}\|_{H^2(\Omega_1 \cup \Omega_2)} + \|p\|_{H^1(\Omega_1 \cup \Omega_2)}).
	\end{align*}
	
	The approximation capability follows since the nonconforming $P_1$ element space proposed by \cite{Wang2025nonconforming} is a subspace of $\bm{U}_h$. Applying the triangle inequality yields the desired result \eqref{ferr}.
\end{proof}

\section{Numerical experiments}\label{sec:example}
The interface is a circle centered at the origin with radius $r=\pi/7,$ i.e.,
\begin{equation*}
	\varPhi_{\Gamma}(x,y) = x^2+y^2-(\pi/7)^2.
\end{equation*}
Let 
\begin{equation*}
	\theta = \varPhi_{\Gamma}^2\, (x-1)^2(y-1)^2.
\end{equation*}
Using de Rham sequence, we construct a divergence free velocity
\begin{equation*}
	\boldsymbol{u} = \frac{1}{\mu}\mathbf{curl}\,\theta
	= \frac{1}{\mu}\left(\partial_y\theta,-\partial_x\theta\right).
\end{equation*}
Since $\partial_x \theta = \partial_y\theta$ on $\Gamma$ , $\boldsymbol{u}$ satisfies the interface condition $[\![\boldsymbol{u}]\!]=0$ on  $\Gamma$. The pressure is given by
\begin{equation*}
	p =x.
\end{equation*}

\begin{table}[H]
	\caption{Numerical results for \textbf{Example 1} with $\mu_1 = 10000,\,\mu_2=1$. } \label{tb:6.1.1}
	\begin{tabular}{c c c c c c c}
		\toprule
		\text{$\frac{1}{h}$}&
		\text{$\|p-p_h\|_{L^2(\Omega)}$}&
		\text{order}&
		\text{$\|\boldsymbol{u}-\boldsymbol{u}_h\|_{L^2(\Omega)}$}&
		\text{order}&
		\text{$|\boldsymbol{u}-\boldsymbol{u}_h|_{H^1(\Omega)}$}&
		\text{order}\\
		\midrule
		16 & 8.0045e-2 &  & 7.7868e-4 &  &  
		4.7226e-2 & \\
		32 & 4.0290e-2 & 0.9903 & 2.0514e-4 & 1.9243 & 
		2.4227e-2 & 0.9629\\
		64 & 2.0166e-2 & 0.9984 & 5.2293e-5 & 1.9719 & 
		1.2179e-2 & 0.9921\\
		128 & 1.0087e-2 & 0.9995 & 1.3095e-5 & 1.9976 & 
		6.1098e-3 & 0.9952\\
		256 & 5.0438e-3 & 0.9998 & 3.2878e-6 & 1.9938 &
		3.0586e-3 & 0.9982\\	
		\bottomrule
	\end{tabular}
\end{table}

\begin{figure}[H]
	\begin{minipage}{0.45\linewidth}        	
		\centering
		\includegraphics[scale=0.4]{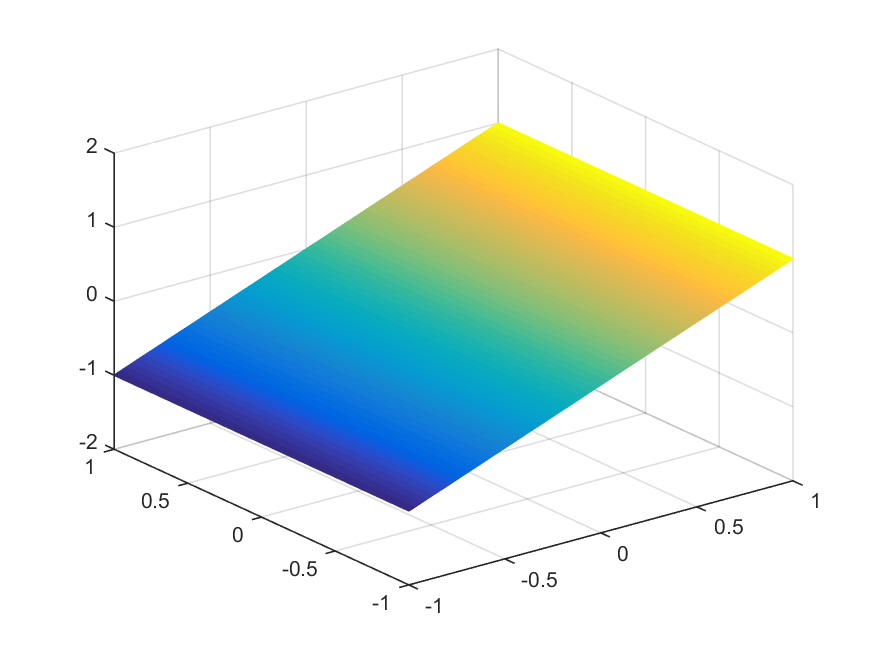}      	
	\end{minipage}
	\begin{minipage}{0.45\linewidth}
		\centering
		\includegraphics[scale=0.4]{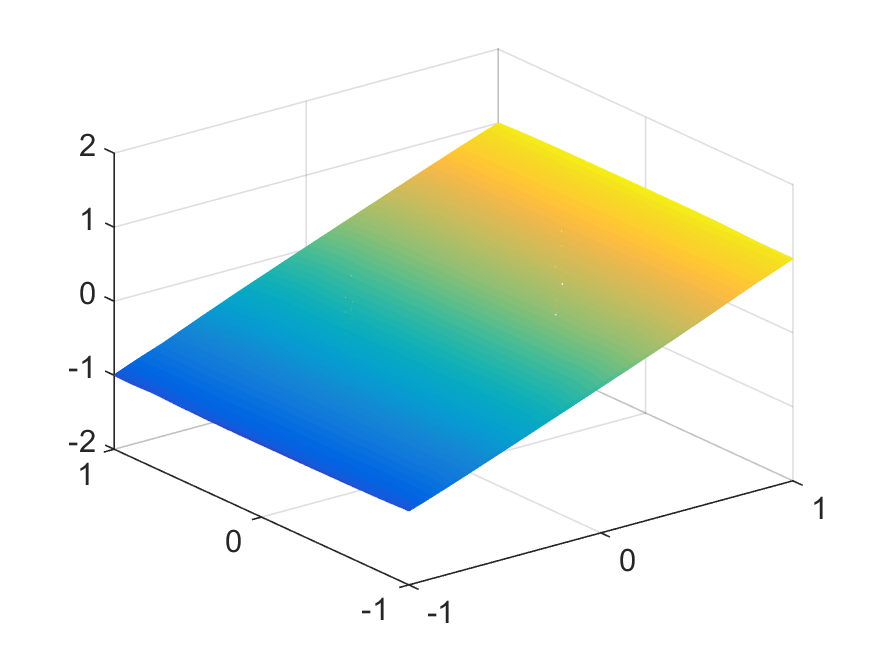}
	\end{minipage}
	\caption{The true solution $p$ (left) and the numerical solution $p_h$ (right) in \textbf{Example 1} with $\mu_1=10000,\,\mu_2=1$. 
	} \label{fig:ex2 p 10000,1}
\end{figure}

\begin{figure}[H]
	\begin{minipage}{0.45\linewidth}        	
		\centering
		\includegraphics[scale=0.4]{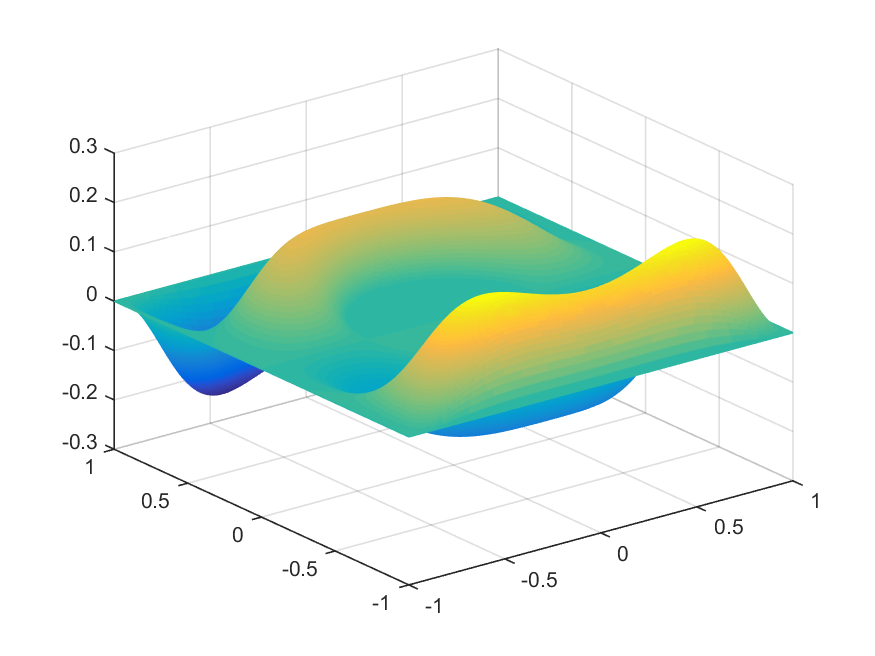}      	
	\end{minipage}
	\begin{minipage}{0.45\linewidth}
		\centering
		\includegraphics[scale=0.4]{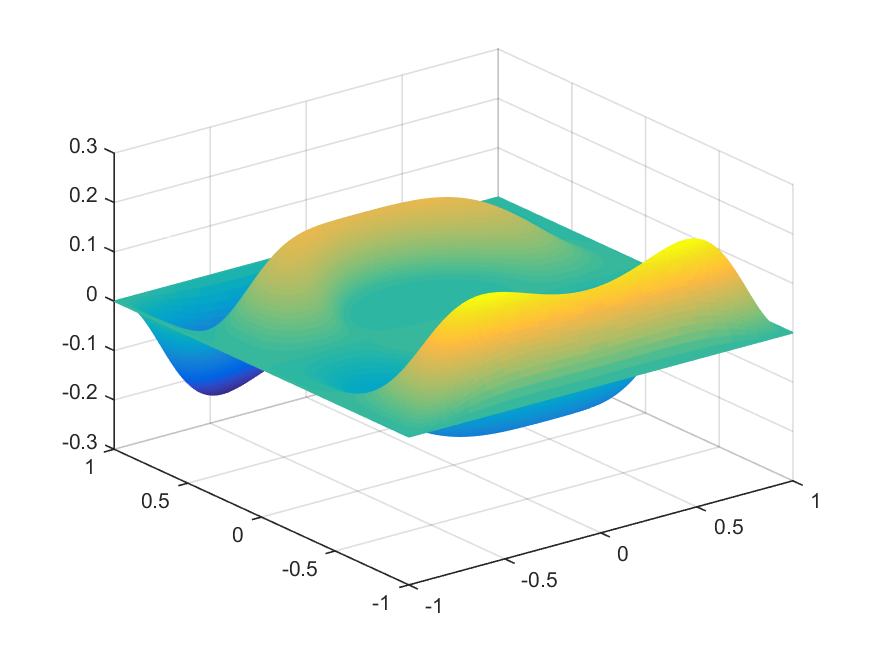}
	\end{minipage}
	\caption{The true solution (left) and the numerical solution (right) of the first component of the velocity field in \textbf{Example 1} with $\mu_1=10000,\,\mu_2=1$.} \label{fig:ex2 u1 10000,1}
\end{figure}

\begin{figure}[H]
	\begin{minipage}{0.45\linewidth}        	
		\centering
		\includegraphics[scale=0.4]{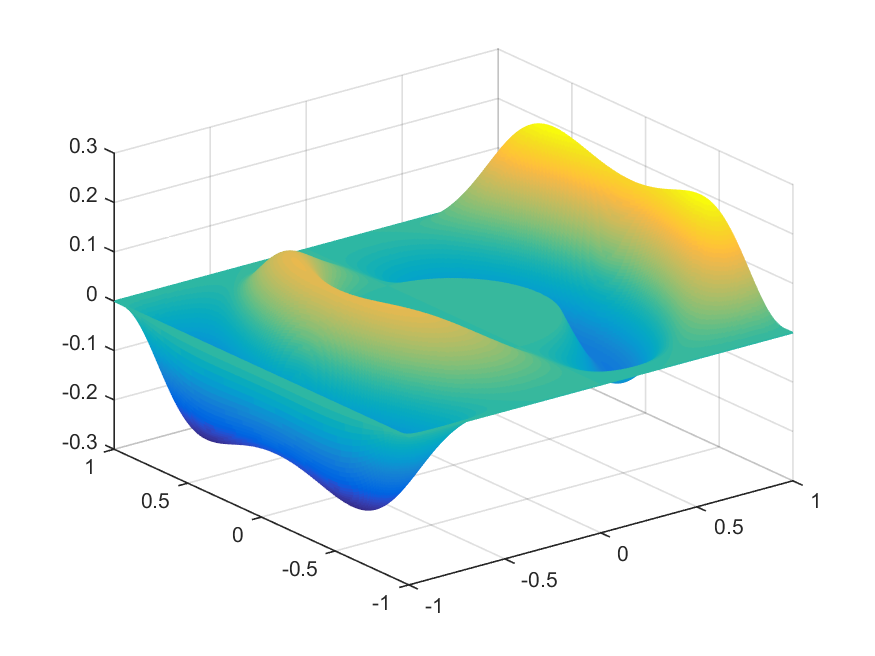}      	
	\end{minipage}
	\begin{minipage}{0.45\linewidth}
		\centering
		\includegraphics[scale=0.4]{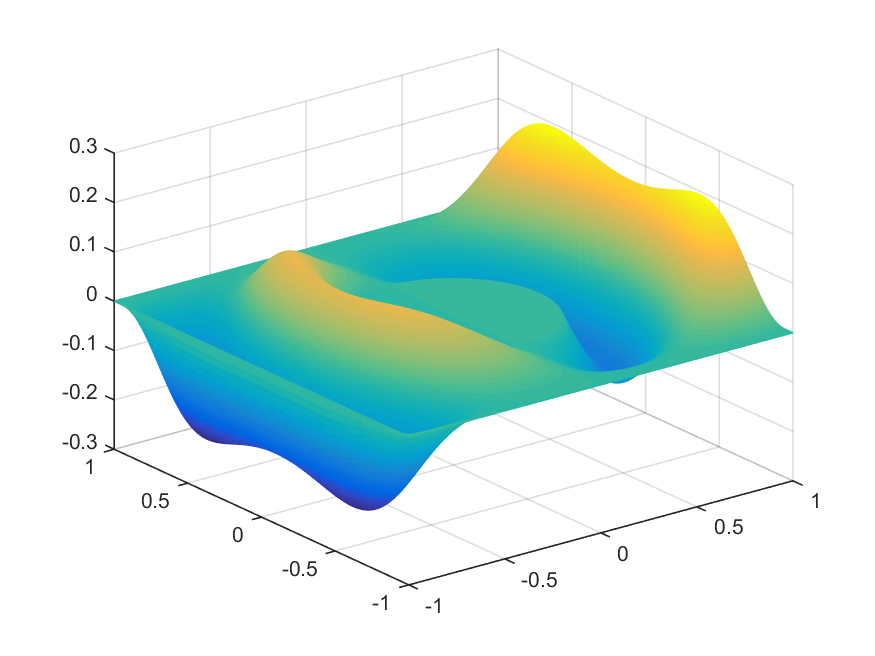}
	\end{minipage}
	\caption{The true solution (left) and the numerical solution (right) of the second component of the velocity field in \textbf{Example 1} with $\mu_1=10000,\,\mu_2=1$.} \label{fig:ex2 u2 10000,1}
\end{figure}

\begin{table}[H]
	\caption{Numerical results for \textbf{Example 1} with $\mu_1 = 1,\,\mu_2=10000$. } \label{tb:6.1.2}
	\begin{tabular}{c c c c c c c}
		\toprule
		\text{$\frac{1}{h}$}&
		\text{$\|p-p_h\|_{L^2(\Omega)}$}&
		\text{order}&
		\text{$\|\boldsymbol{u}-\boldsymbol{u}_h\|_{L^2(\Omega)}$}&
		\text{order}&
		\text{$|\boldsymbol{u}-\boldsymbol{u}_h|_{H^1(\Omega)}$}&
		\text{order}\\
		\midrule
		16 & 1.2287e-1 &  & 1.0368e-2 &  &  
		4.4406e-1 & \\
		32 & 5.1387e-2 & 1.2576 & 2.4984e-3 & 2.0531 & 
		2.2395e-1 & 0.9875\\
		64 & 2.0165e-2 & 1.3494 & 6.1676e-4 & 2.0182 & 
		1.1182e-1 & 1.0019\\
		128 & 1.1291e-2 & 0.8366 & 1.5628e-4 & 1.9805 & 
		5.6205e-2 & 0.9924\\
		256 & 5.3618e-3 & 1.0744 & 3.8967e-5 & 2.0038 &
		2.8069e-2 & 1.0017\\
		\bottomrule
	\end{tabular}
\end{table}

\begin{figure}[H]
	\begin{minipage}{0.45\linewidth}        	
		\centering
		\includegraphics[scale=0.4]{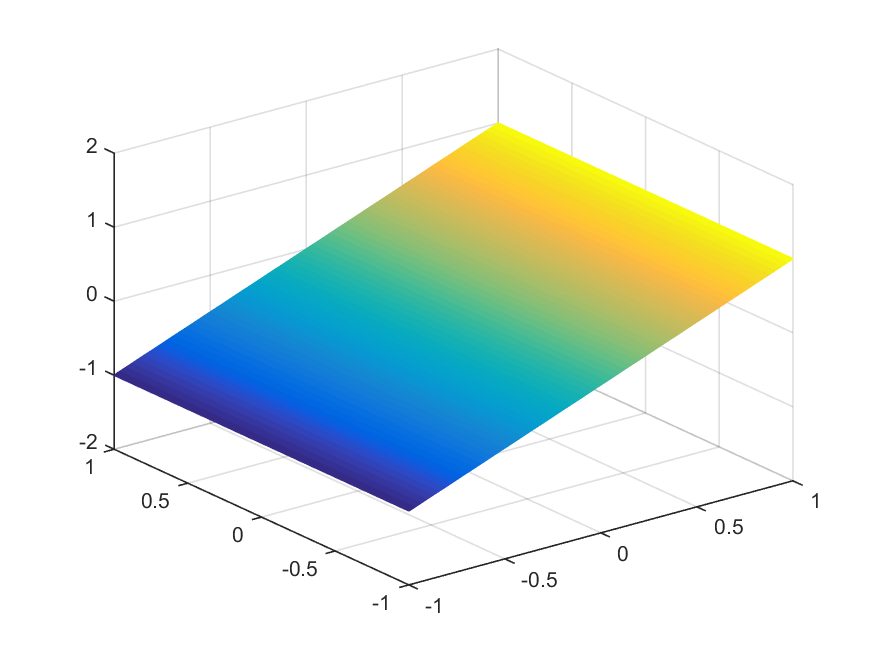}      	
	\end{minipage}
	\begin{minipage}{0.45\linewidth}
		\centering
		\includegraphics[scale=0.4]{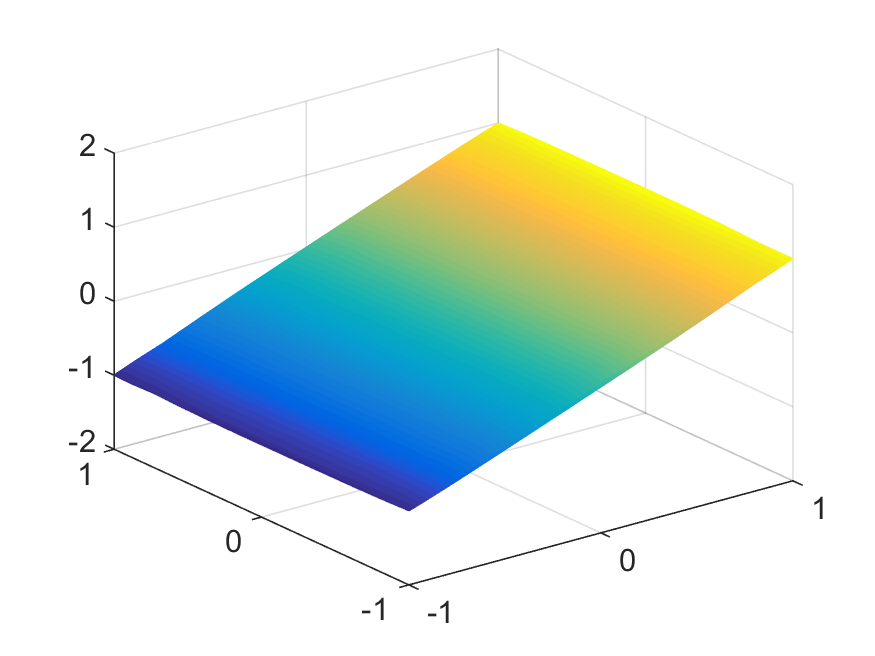}
	\end{minipage}
	\caption{The true solution $p$ (left) and the numerical solution $p_h$ (right) in \textbf{Example 1} with $\mu_1=1,\,\mu_2=10000$.} \label{fig:ex2 p 1,10000}
\end{figure}

\begin{figure}[H]
	\begin{minipage}{0.45\linewidth}        	
		\centering
		\includegraphics[scale=0.4]{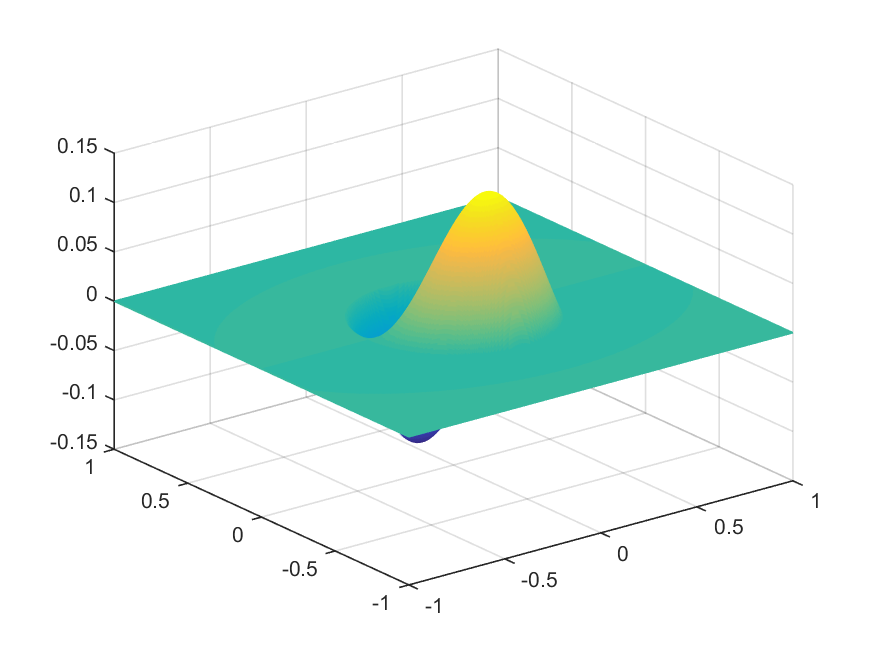}      	
	\end{minipage}
	\begin{minipage}{0.45\linewidth}
		\centering
		\includegraphics[scale=0.4]{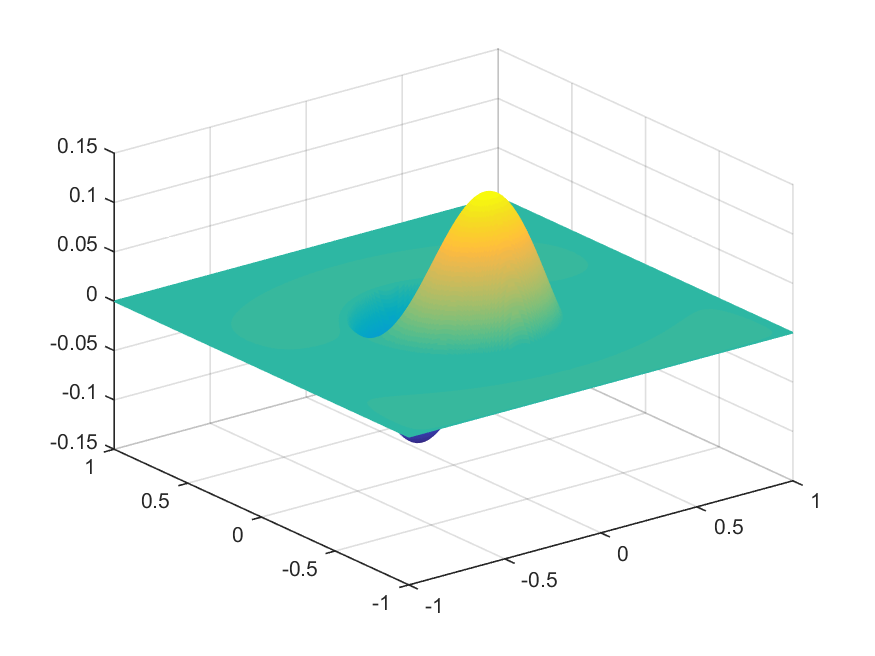}
	\end{minipage}
	\caption{The true solution (left) and the numerical solution (right) of the first component of the velocity field in \textbf{Example 1} with $\mu_1=1,\,\mu_2=10000$.} \label{fig:ex2 u1 1,10000}
\end{figure}

\begin{figure}[H]
	\begin{minipage}{0.45\linewidth}        	
		\centering
		\includegraphics[scale=0.4]{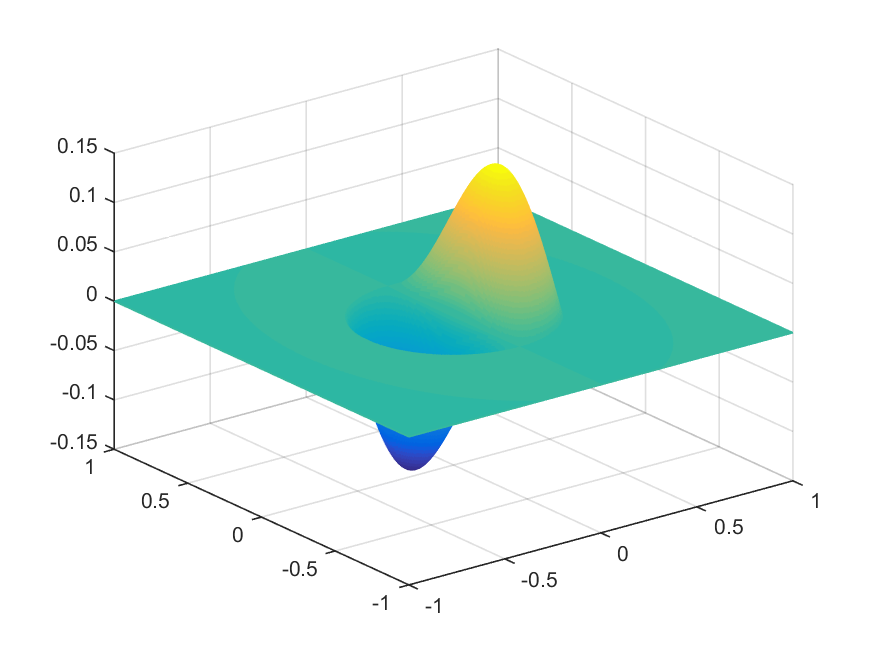}      	
	\end{minipage}
	\begin{minipage}{0.45\linewidth}
		\centering
		\includegraphics[scale=0.4]{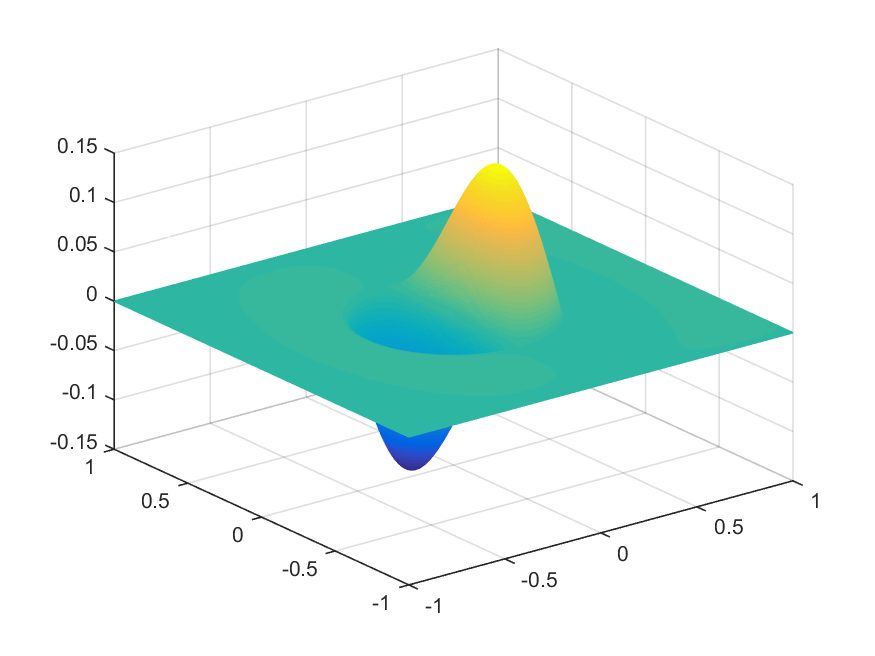}
	\end{minipage}
	\caption{The true solution (left) and the numerical solution (right) of the second component of the velocity field in \textbf{Example 1} with $\mu_1=1,\,\mu_2=10000.$} \label{fig:ex2 u2 1,10000}
\end{figure}

\begin{table}[H]
	\caption{Numerical results for \textbf{Example 1} with $\mu_1 = 1,\,\mu_2=100$. } \label{tb:6.1.3}
	\begin{tabular}{c c c c c c c}
		\toprule
		\text{$\frac{1}{h}$}&
		\text{$\|p-p_h\|_{L^2(\Omega)}$}&
		\text{order}&
		\text{$\|\boldsymbol{u}-\boldsymbol{u}_h\|_{L^2(\Omega)}$}&
		\text{order}&
		\text{$|\boldsymbol{u}-\boldsymbol{u}_h|_{H^1(\Omega)}$}&
		\text{order}\\
		\midrule
		16 & 8.3069e-2 &  & 9.6682e-3 &  &  
		4.3851e-1 & \\
		32 & 4.1260e-2 & 1.0095 & 2.4582e-3 & 1.9756 & 
		2.2292e-1 & 0.9761\\
		64 & 2.0165e-2 & 1.0328 & 6.1672e-4 & 1.9949 & 
		1.1182e-1 & 0.9953\\
		128 & 1.0134e-2 & 0.9926 & 1.5449e-4 & 1.9971 & 
		5.5997e-2 & 0.9977\\
		256 & 5.0524e-3 & 1.0041 & 3.8618e-5 & 2.0001 &
		2.7998e-2 & 0.9999\\
		\bottomrule
	\end{tabular}
\end{table}

\begin{table}[H]
	\caption{Numerical results for \textbf{Example 1} with $\mu_1 = 100,\,\mu_2=1.$ } \label{tb:6.1.4}
	\begin{tabular}{c c c c c c c}
		\toprule
		\text{$\frac{1}{h}$}&
		\text{$\|p-p_h\|_{L^2(\Omega)}$}&
		\text{order}&
		\text{$\|\boldsymbol{u}-\boldsymbol{u}_h\|_{L^2(\Omega)}$}&
		\text{order}&
		\text{$|\boldsymbol{u}-\boldsymbol{u}_h|_{H^1(\Omega)}$}&
		\text{order}\\
		\midrule
		16 & 8.0033e-2 &  & 7.8541e-4 &  &  
		4.7412e-2 & \\
		32 & 4.0285e-2 & 0.9903 & 2.0693e-4 & 1.9243 & 
		2.4325e-2 & 0.9627\\
		64 & 2.0166e-2 & 0.9983 & 5.2732e-5 & 1.9723 & 
		1.2230e-2 & 0.9919\\
		128 & 1.0086e-2 & 0.9995 & 1.3206e-5 & 1.9974 & 
		6.1353e-3 & 0.9952\\
		256 & 5.0437e-3 & 0.9998 & 3.3152e-6 & 1.9940 &
		3.0714e-3 & 0.9982\\
		\bottomrule
	\end{tabular}
\end{table}
\section{Conclusion}\label{sec:conclusion}
%因为网格中的四边形是可能退化的，对应的the rotated $Q_1$ type element的二次项的选择
In this work, we proposed a new nonconforming finite element method for solving the Stokes interface problems. The method was constructed on a local anisotropic hybrid mesh, which was first introduced in our earlier work \cite{Hu2021optimal}. The present results further demonstrate the effectiveness of this type of mesh in accurately resolving interface geometry while maintaining computational simplicity. The proposed nonconforming element reduces to the standard Crouzeix–Raviart element on triangular elements and to a new rotated $Q_1$ - type element on quadrilateral elements. This structure naturally accommodates the use of hybrid meshes and may be beneficial in other applications where elements of different shapes need to be effectively coupled. The consistency error is of optimal convergence order, as proved in our previous paper \cite{Wang2025nonconforming}. More importantly, we proved that this element satisfies the inf - sup condition without any stabilization terms, which is quite rare in the existing literature on finite element methods for Stokes interface problems. 

\section*{Declarations}
\textbf{Conflict of interest} The authors declare that they have no conflict of interest.
\bibliographystyle{plain}        % 支持eprint的样式
\bibliography{refer}
\end{document}